\documentclass{article}

\usepackage{mathrsfs}
\usepackage[linesnumbered,ruled,vlined]{algorithm2e} 
\usepackage{a4wide}
\usepackage{amsmath,amssymb,amsthm}
\usepackage{graphicx}
\usepackage{url}

\newtheorem{Theorem}{Theorem}
\newtheorem{Definition}{Definition}

\begin{document}

\title{Dynamic Analysis and Optimal Prevention Strategies 
for Monkeypox Spread Modeled via the Mittag--Leffler Kernel\thanks{This is 
a preprint of a paper published in 'Fractal Fract.' 
at [https://doi.org/10.3390/fractalfract10010044].}}

\author{Mine Yurto\u{g}lu$^{1}$\\
{\tt mine.yurtoglu91@gmail.com}
\and Dilara Yap\i \c{s}kan$^{1,2}$\\
{\tt dilarayapiskan@ua.pt}
\and Ebenezer Bonyah$^{3}$\\
{\tt ebonyah@aamusted.edu.gh}
\and Beyza Billur \.{I}skender Ero\u{g}lu$^{1}$\\
{\tt biskender@balikesir.edu.tr}
\and Derya Avc\i \,$^{1}$\\
{\tt dkaradeniz@balikesir.edu.tr}
\and Delfim F. M. Torres $^{2,}$\thanks{Correspondence: delfim@ua.pt}\\
{\tt delfim@ua.pt}}

\date{$^{1}$Department of Mathematics, Bal\i kesir University, 10145 Bal\i kesir, Turkey\\[0.3cm]
$^{2}$\text{Center for Research and Development in Mathematics and Applications (CIDMA),}
Department of Mathematics, University of Aveiro, 3810-193 Aveiro, Portugal\\[0.3cm]
$^{3}$Department of Mathematics Education, 
Akenten Appiah Menka University of Skills Training and Entrepreneurial Development, 
Kumasi P.O. Box 1277, Ghana}

\maketitle


\begin{abstract}
Monkeypox is a viral disease belonging to the smallpox family. Although it has 
milder symptoms than smallpox in humans, it has become a global threat in recent 
years, especially in African countries. Initially, incidental immunity against 
monkeypox was provided by smallpox vaccines. However, the eradication of smallpox 
over time and thus the lack of vaccination has led to the widespread and clinical 
importance of monkeypox. Although mathematical epidemiology research on the disease 
is complementary to clinical studies, it has attracted attention in the last few years. 
The present study aims to discuss the indispensable effects of three control strategies 
such as vaccination, treatment, and quarantine to prevent the monkeypox epidemic modeled 
via the Atangana--Baleanu operator. The main purpose is to determine optimal control 
measures planned to reduce the rates of exposed and infected individuals at the minimum 
costs. For the controlled model, the existence-uniqueness of the solutions, stability, 
and sensitivity analysis, and numerical optimal solutions are exhibited. The optimal 
system is numerically solved using the Adams-type predictor--corrector method. 
In the numerical simulations, the efficacy of the vaccination, treatment, and quarantine 
controls is evaluated in separate analyzes as single-, double-, and triple-control 
strategies. The results demonstrate that the most effective strategy for achieving the 
aimed outcome is the simultaneous application of vaccination, treatment, and quarantine 
controls.

\medskip


\medskip

\end{abstract}


\section{Introduction}

The Poxviridae family consists of double-stranded DNA viruses that have long
plagued animals and humans. Monkeypox virus belonging to this family is a
viral zoonotic disease and was first identified in the mid-19th century. It
has since been a serious health problem in West and Central Africa. Since
the monkeypox virus belongs to the smallpox family,
vaccines used for smallpox have provided protection against monkeypox for
many years~\cite{WHO1}. However, with~the eradication of smallpox, the~
vaccine has not become a need, and~the cases have increased dramatically,
making it a threat today. As~a result, in~July 2022, the~World Health
Organization (WHO) declared the public health emergency as the global
monkeypox outbreak, announcing that the monkeypox virus continued to spread
from 75 countries and regions, causing international concern~\cite{WHO2,Hatmal}. 
This critical change in monkeypox has drawn attention not
only in medicine but also in mathematical epidemiology. In~this context,
mathematical models representing different variations of the disease have
been introduced, because~mathematical models provide an opportunity to make
a realistic prediction about the future state of a disease. In~addition,
thanks to mathematical models, the~optimal prevention strategies that save
both time and cost can be determined to neutralize the~disease.

Mathematical modeling of a disease dates back to Bernoulli's work on the
spread of the Smallpox virus in the 18th century~\cite{Bernoulli}. In~the
first quarter of the next century, Kermack and McKendrick~\cite{Kermack}
modeled the spread dynamics of an infectious diseases with the help of a
differential equation system. Nowadays, mathematical modeling has an
increasing interest among the researchers to foresee the precautions for
reducing the spread of diseases such as Dengue, HIV, 
and~COVID-19~\cite{torres1,Bolaji,Baskonus}. 
Although~monkeypox has also posed a critical
threat around the world in recent years, little attention has been paid to
the introduction of its mathematical models and preventive control
strategies~\cite{Bhunu1,Bhunu2,Usman,Emeka,Peter1,Monkeypox2023}.

The classical derivative is, unfortunately, inadequate due to its local
definition in modeling the heterogeneous propagation behavior occurring in a
physical process or a biological phenomenon. The~most important reason for
this is that many external factors disrupt the homogeneity of behavior. This
shortcoming is easily remedied with fractional derivatives, thanks to their
non-local definitions. As~known, fractional derivatives are classified as
singular and non-singular, according to their kernel structure. This makes
one~advantageous over the other depending on the discussed 
phenomenon~\cite{Baltaeva,Muhammad,Tajani,Sabatier}. {Especially, a~review of the 
literature reveals that the Atangana--Baleanu fractional derivative, thanks 
to its Mittag--Leffler kernel, ensures better insight into the evolution of the 
disease, particularly at the beginning and end of the spread of 
viral diseases~\cite{Ullah,Li}.}

In epidemiological modeling, although~it is usually assumed that the
infection spreads homogeneously in the population, the~transmission depends
on the diversity of individual characteristics (age, gender, physical
characteristics, genetics, etc.) and environmental factors (habitat,
population density, financial opportunities, technological developments,
etc.). Therefore, in~reality, many infectious diseases show a heterogeneous
spread~\cite{Becker}. In~this sense, fractional operators are used to obtain
the closest model to reality~\cite{FatmaDerya,Yurtoglu2024,BeyzaDilara,Yapiskan2024}.
Smallpox displays an exponential distribution for the reasons mentioned~\cite{Kaplan}. 
Since monkeypox belongs to the same family and spreads among
different populations, it shows an inhomogeneous distribution. Thus,
monkeypox has been discussed very recently with fractional-order 
models~\cite{Peter2,El-Mesady,Okyere,Okposo2023,Zhang2023}.

Optimal control theory is a powerful mathematical tool that complements the
system theory, because~the determination of the parameters that optimally
control the behavior of the systems is quite important, as is examining
the dynamics of the systems. In~fractional optimal control theory, the~
system and/or objective function is defined in terms of fractional
operators. Determination of the optimality conditions depends on the type of 
fractional operator~\cite{Agrawal1,Agrawal2,Bahaa19}. 
To our knowledge, optimal control studies of integer or fractional order 
monkeypox models are just limited in the literature~\cite{Majee23,Petercontrol}. 
Motivated by the need for optimal control strategies to prevent the spread of 
monkeypox, we develop the model of Peter~et~al.~\cite{Peter1} in the
present study. The~controlled model is discussed in the sense of
Atangana--Baleanu derivative. Vaccination, treatment, and~quarantine are 
validated as control strategies. These controls are also addressed for the 
elimination of different diseases~\cite{Torres2,Eroglu2023,Yurtoglu2023,Ammi}.

In the rest of the study, basic definitions and properties are
recalled in\linebreak   Section~\ref{sec:2}. The~controlled model is proposed in Section~\ref{sec:3}. 
The non-negativity, boundedness, existence and uniqueness of the solutions are
proved in~Section~\ref{sec:4}. 
In Section~\ref{sec:5}, the~reproduction number is computed for the controlled system. 
Afterward, the~local stability of the model is analyzed. The~developed optimal control 
problem aims to minimize both the number of exposed and infected individuals, and~the 
costs of all control strategies. The~existence of controls is first guaranteed and then 
the optimal control problem is constructed in Section~\ref{sec:6}. Thereafter, by~
Pontryagin's Maximum Principle, the~fractional necessary optimal system is achieved. 
Finally, this system is solved with the Adams-type predictor--corrector algorithm combined with the
forward--backward sweep method. In~Section~\ref{sec:7}, the~behaviors of the system under controls
are implemented by giving various graphics, which are plotted with the help of MATLAB 2021b 
software. We end with Section~\ref{sec:8}, which presents the~conclusions.


\section{Fundamental~Definitions}
\label{sec:2}

This section provides necessary fundamental definitions of fractional
operators\linebreak   as follows:

\begin{Definition}
The one and two parameter Mittag-Leffler functions are respectively
defined as~\cite{Kilbas93}: 
\vspace{-3pt}
\begin{eqnarray}
E_{\omega }(z) &=&\sum_{i=0}^{\infty }\frac{z^{i}}{\Gamma (\omega i+1)}, \Re
(\omega ) \in \mathbb{R} _{+},  \label{eq7} \\
E_{\omega,\xi }(z) &=&\sum_{i=0}^{\infty }\frac{z^{i}}{\Gamma (\omega i+\xi
)}, \Re (\omega ) \in \mathbb{R} _{+}, \Re (\xi ) \in \mathbb{R} _{+},
\label{eq8}
\end{eqnarray}
in where $\Gamma(\cdot)$ is the Gamma function.
\end{Definition}

\begin{Definition}
Let $f\left( t\right) \in C\left( \left[ a,b\right] \right) $, $t\in \left(
a,b\right) $, and~$\alpha \in \mathbb{R}$. 
Then the Riemann-Liouville (RL) fractional integral is defined as follows
\cite{Kilbas93}:
\begin{equation}
^{RL}I_{a+}^{\alpha }f\left( t\right) =\frac{1}{\Gamma \left( \alpha \right) 
}\underset{a}{\overset{\tau }{\int }}f\left( t\right) \left( \tau -t\right)
^{\alpha -1}dt.  \label{eq1}
\end{equation}
\end{Definition}

\begin{Definition}
The Atangana- Baleanu (AB) integral of $f$ is defined as~\cite{BalFer18}:
\begin{equation}
^{AB}I_{0+}^{\alpha }f(t)=\frac{1-\alpha }{B(\alpha )}+\frac{\alpha }{
B(\alpha )}^{RL}I_{0+}^{\alpha }f\left( t\right), \alpha \in (0,1).  \label{eq9a}
\end{equation}
\end{Definition}

\begin{Definition}
For $0\leq \alpha \leq 1$ and $f\in H^{1}(t_{0},t_{f}),$ the $\alpha $-order
left and right AB fractional derivative in the Caputo sense is defined as~\cite{Atangana2016}:
\begin{eqnarray}
_{t_{0}}^{ABC}D_{t}^{\alpha }f(t) &=&\frac{B(\alpha )}{1-\alpha }
\int_{t_{0}}^{t}E_{\alpha }\left[ -\frac{\alpha }{1-\alpha }(t-\theta
)^{\alpha }\right] f^{\prime }(\theta )d\theta,  \label{eq5} \\
_{t}^{ABC}D_{t_{f}}^{\alpha }f(t) &=&-\frac{B(\alpha )}{1-\alpha }
\int_{t}^{t_{f}}E_{\alpha }\left[ -\frac{\alpha }{1-\alpha }(\theta
-t)^{\alpha }\right] f^{\prime }(\theta )d\theta,  \label{eq6}
\end{eqnarray}
where $B(\alpha )$ is satisfying $B(0)=B(1)=1$.
\end{Definition}

\begin{Definition}
The Laplace transformation of AB derivative is given by~\cite{Atangana2016}:
\begin{equation}
\mathscr{L} \left\{ _{0}^{ABC}D_{t}^{\alpha }f(t)\right\} (s)
=\frac{B(\alpha)}{1-\alpha }\frac{s^{\alpha }\mathscr{L} 
\left\{ f(t)\right\} (s)-s^{\alpha}f(0)}{s^{\alpha }+\frac{\alpha }{1-\alpha }}.  
\label{eq9}
\end{equation}
\end{Definition}

\begin{Definition}
{Consider the fractional-order system 
${}^{ABC}_{0}D_t^{\alpha}x(t)=F(x(t))$ with $0<\alpha<1$.
An equilibrium point $x$ is called locally 
asymptotically stable if solutions that start sufficiently
close to $x$ remain close to $x$ and satisfy $\|x(t)-x\| \to 0$ as $t\to\infty$.}
\end{Definition}


\section{Model~Formulation}
\label{sec:3}

The basis model of the present study consists of humans and rodents. It
describes the interspecies spread behavior of monkeypox. In~this context,
the humans are classified as susceptible $S_{h}(t)$, exposed $E_{h}(t)$,
infected $I_{h}(t)$, isolated $Q_{h}(t),$ and recovered $R_{h}(t)$.
Similarly, the~rodent population includes susceptible $S_{r}(t)$, exposed 
$E_{r}(t)$ and infected $I_{r}(t)$ rodents. Thus, the~total human and rodent
populations are denoted as $N_{h}(t)$ and $N_{r}(t)$, respectively.
Migrating or newborn humans are supplemented to the $S_{h}(t)$ at the rate
of $\theta _{h}$, and~$\theta _{r}$ is the participation rate for 
$S_{r}(t)$. Susceptible humans can be exposed to monkeypox by interacting
with infected rodents at a rate $\ \beta _{1}$ and with infected humans at a
rate $\beta _{2}$. Also, susceptible rodents can be exposed to monkeypox by
interacting with infected rodents at a rate $\beta _{3}$. The~rate of
exposed human to infected humans is $\alpha _{1}.$ The rate defined as
suspicious case is $\alpha _{2}.$ The rate of undetected after medical
diagnosis is $\varphi $. The~rate of progression from the isolated class to
the recovered class is $\tau $. The~recovery rate for human is $\gamma $.
The natural death rate of humans is $\mu _{h},$ the natural mortality rate
of rodents is $\mu _{r}.$ By $\delta _{h}$ we represent the disease-related mortality
for humans, while $\delta _{r}$ represents disease-related mortality for
rodents.

The non-linear ordinary differential equation system resulting from the
interactions between the compartments is as follows:
{\begin{equation}
\left. 
\begin{array}{l}
\frac{dS_{h}\left( t\right) }{dt}=\theta _{h}-\frac{(\beta _{1}I_{r}\left(
t\right) +\beta _{2}I_{h}\left( t\right) )S_{h}\left( t\right) }{N_{h}\left(
t\right) }-\mu _{h}S_{h}\left( t\right) +\varphi Q_{h}\left( t\right) , \\ 
\frac{dE_{h}\left( t\right) }{dt}=\frac{(\beta _{1}I_{r}\left( t\right)
+\beta _{2}I_{h}\left( t\right) )S_{h}\left( t\right) }{N_{h}\left( t\right) 
}-(\alpha _{1}+\alpha _{2}+\mu _{h})E_{h}\left( t\right) , \\ 
\frac{dI_{h}\left( t\right) }{dt}=\alpha _{1}E_{h}\left( t\right) -(\mu
_{h}+\delta _{h}+\gamma )I_{h}\left( t\right) , \\ 
\frac{dQ_{h}\left( t\right) }{dt}=\alpha _{2}E_{h}\left( t\right) -(\varphi
+\tau +\mu _{h}+\delta _{h})Q_{h}\left( t\right) , \\ 
\frac{dR_{h}\left( t\right) }{dt}=\gamma I_{h}\left( t\right) +\tau
Q_{h}\left( t\right) -\mu _{h}R_{h}\left( t\right) , \\ 
\frac{dS_{r}\left( t\right) }{dt}=\theta _{r}-\frac{\beta _{3}S_{r}\left(
t\right) I_{r}\left( t\right) }{N_{r}\left( t\right) }-\mu _{r}S_{r}\left(
t\right) , \\ 
\frac{dE_{r}\left( t\right) }{dt}=\frac{\beta _{3}S_{r}\left( t\right)
I_{r}\left( t\right) }{N_{r}\left( t\right) }-(\mu _{r}+\alpha
_{3})E_{r}\left( t\right) , \\ 
\frac{dI_{r}\left( t\right) }{dt}=\alpha _{3}E_{r}\left( t\right) -(\mu
_{r}+\delta _{r})I_{r}\left( t\right) .
\end{array}
\right\}   
\label{eq10}
\end{equation}}

The discussed integer-order model $\left( \ref{eq10}\right) $ was first
proposed by Peter~et~al.~\cite{Peter1} to understand the dynamics of the
monkeypox virus. Unfortunately, integer-order models have limitations such as
not being able to represent the memory and inheritance characteristics of
systems due to their local definition. Although~RL and Caputo derivatives,
which are the leading operators of fractional calculus, have been successful
in eliminating this deficiency, regular fractional operators have been
introduced that overcome various difficulties arising from the singular
structures of RL and Caputo operators. For~this aim, in~2016, Atangana and
Baleanu proposed a new fractional operator with Mittag-Leffler 
kernel~\cite{Atangana2016}. Since the AB fractional derivative does not pose a
singularity problem at the onset and the end of the disease due to its
definition, it provides a better idea of the state of the disease at these
critical moments. This is why the system analysis and optimal control of the
model considered in this study are examined with the AB fractional
derivative. We adapt three control functions to the system in accordance
with the model dynamics. In~the controlled model, {the control functions $u_1(\cdot)$, 
$u_2(\cdot)$, and~$u_3(\cdot)$} represent vaccination of susceptible humans, 
treatment of infected humans, and~quarantine of infected humans, respectively. 
As a result, the~model in which unit consistency is ensured is written as follows: 
\begin{equation}
\left. 
\begin{array}{l}
_{0}^{ABC}D_{t}^{\alpha }S_{h}\left( t\right) =\theta _{h}^{\alpha }-\frac{
(\beta _{1}^{\alpha }I_{r}\left( t\right) +\beta _{2}^{\alpha }I_{h}\left(
t\right) )S_{h}\left( t\right) }{N_{h}\left( t\right) }-\mu _{h}^{\alpha
}S_{h}\left( t\right) +\varphi ^{\alpha }Q_{h}\left( t\right) -u_{1}\left(
t\right) S_{h}\left( t\right) , \\ 
_{0}^{ABC}D_{t}^{\alpha }E_{h}\left( t\right) =\frac{(\beta _{1}^{\alpha
}I_{r}\left( t\right) +\beta _{2}^{\alpha }I_{h}\left( t\right) )S_{h}\left(
t\right) }{N_{h}\left( t\right) }-(\alpha _{1}^{\alpha }+\alpha _{2}^{\alpha
}+\mu _{h}^{\alpha })E_{h}\left( t\right) , \\ 
_{0}^{ABC}D_{t}^{\alpha }I_{h}\left( t\right) =\alpha _{1}^{\alpha
}E_{h}\left( t\right) -(\mu _{h}^{\alpha }+\delta _{h}^{\alpha }+\gamma
^{\alpha })I_{h}\left( t\right) -u_{2}\left( t\right) I_{h}\left( t\right)
-u_{3}\left( t\right) I_{h}\left( t\right) , \\ 
_{0}^{ABC}D_{t}^{\alpha }Q_{h}\left( t\right) =\alpha _{2}^{\alpha
}E_{h}\left( t\right) -(\varphi ^{\alpha }+\tau ^{\alpha }+\mu _{h}^{\alpha
}+\delta _{h}^{\alpha })Q_{h}\left( t\right) +u_{3}\left( t\right)
I_{h}\left( t\right) , \\ 
_{0}^{ABC}D_{t}^{\alpha }R_{h}\left( t\right) =\gamma ^{\alpha }I_{h}\left(
t\right) +\tau ^{\alpha }Q_{h}\left( t\right) -\mu _{h}^{\alpha }R_{h}\left(
t\right) +u_{1}\left( t\right) S_{h}\left( t\right) +u_{2}\left( t\right)
I_{h}\left( t\right) , \\ 
_{0}^{ABC}D_{t}^{\alpha }S_{r}\left( t\right) =\theta _{r}^{\alpha }-\frac{
\beta _{3}^{\alpha }S_{r}\left( t\right) I_{r}\left( t\right) }{N_{r}\left(
t\right) }-\mu _{r}^{\alpha }S_{r}\left( t\right) , \\ 
_{0}^{ABC}D_{t}^{\alpha }E_{r}\left( t\right) =\frac{\beta _{3}^{\alpha
}S_{r}\left( t\right) I_{r}\left( t\right) }{N_{r}\left( t\right) }-(\mu
_{r}^{\alpha }+\alpha _{3}^{\alpha })E_{r}\left( t\right) , \\ 
_{0}^{ABC}D_{t}^{\alpha }I_{r}\left( t\right) =\alpha _{3}^{\alpha
}E_{r}\left( t\right) -(\mu _{r}^{\alpha }+\delta _{r}^{\alpha })I_{r}\left(
t\right), 
\end{array}
\right\}   \label{eq11}
\end{equation}
with given initial conditions
\begin{equation}
\left(S_{h}(0),E_{h}(0),I_{h}(0),Q_{h}(0),R_{h}(0),S_{r}(0),E_{r}(0),I_{r}(0)
\right).  
\label{eq12}
\end{equation}
Also, the~admissible set of controls is defined by
\begin{equation}
U_{ad}=\left\{ \left( u_{1}(\cdot),u_{2}(\cdot),u_{3}(\cdot)\right) \mid 0\leq
u_{1}(t),u_{2}(t),u_{3}(t)\leq 0.9,\text{ }0\leq t\leq t_{f}\right\}.
\label{eq32}
\end{equation}


\section{System~Analysis}
\label{sec:4}

Let us first guarantee positivity and boundedness of
the solution region. These properties are indispensable as
they show that the system is well-defined~biologically.

\subsection{The Feasibility of~Region}

Theorem \ref{Theorem1} proves the non-negativity and boundedness of the solutions:

\begin{Theorem}\label{Theorem1}
The region $\Omega =$ $\Omega _{h}\times \Omega _{r}$ for the model 
\eqref{eq11} such that 
{\begin{equation}
\Omega _{h}=\left\{ \left( S_{h},E_{h},I_{h},Q_{h},R_{h}\right) 
\in \mathbb{R}^{5}:S_{h}\left( t\right), E_{h}\left( t\right), 
I_{h}\left( t\right), Q_{h}\left( t\right), R_{h}\left( t\right)\geq 0\right\}  
\label{eq13a}
\end{equation}
and
\begin{equation}
\Omega _{r}=\left\{ \left( S_{r},E_{r},I_{r}\right) \in \mathbb{R}^{3}:
S_{r}\left( t\right), E_{r}\left( t\right), I_{r}\left( t\right)\geq 0\right\},  
\label{eq13b}
\end{equation}}
is a positive invariant set.
\end{Theorem}

\begin{proof}
For the system $\left( \ref{eq11}\right) $, the~following results are valid: 
\begin{eqnarray*}
_{0}^{ABC}D_{t}^{\alpha }S_{h}\left( t\right) &\mid &_{S_{h}=0}=\theta _{h}^{\alpha
}+\varphi ^{\alpha }Q_{h}\left( t\right)\geq 0, \\
_{0}^{ABC}D_{t}^{\alpha }E_{h} \left( t\right)&\mid &_{E_{h}=0}=\frac{(\beta _{1}^{\alpha
}I_{r}\left( t\right)+\beta _{2}^{\alpha }I_{h}\left( t\right))S_{h}\left( t\right)}{N_{h}\left( t\right)}\geq 0, \\
_{0}^{ABC}D_{t}^{\alpha }I_{h}\left( t\right) &\mid &_{I_{h}=0}=\alpha _{1}^{\alpha
}E_{h}\left( t\right)\geq 0, \\
_{0}^{ABC}D_{t}^{\alpha }Q_{h}\left( t\right) &\mid &_{Q_{h}=0}=\alpha _{2}^{\alpha
}E_{h}\left( t\right)+u_{3}\left( t\right)I_{h}\left( t\right)\geq 0, \\
_{0}^{ABC}D_{t}^{\alpha }R_{h}\left( t\right) &\mid &_{R_{h}=0}=\gamma ^{\alpha }I_{h}\left( t\right)+\tau
^{\alpha }Q_{h}\left( t\right)+u_{1}\left( t\right)S_{h}\left( t\right)+u_{2}\left( t\right)I_{h}\left( t\right)\geq 0, \\
_{0}^{ABC}D_{t}^{\alpha }S_{r}\left( t\right) &\mid &_{S_{r}=0}=\theta _{r}^{\alpha }\geq 0, \\
_{0}^{ABC}D_{t}^{\alpha }E_{r}\left( t\right) &\mid &_{E_{r}=0}=\frac{\beta _{3}^{\alpha
}S_{r}\left( t\right)I_{r}\left( t\right)}{N_{r}\left( t\right)}\geq 0, \\
_{0}^{ABC}D_{t}^{\alpha }I_{r}\left( t\right) &\mid &_{I_{r}=0}=\alpha _{3}^{\alpha
}E_{r}\left( t\right)\geq 0.
\end{eqnarray*}
It follows that all solutions of $\left( \ref{eq11}\right) $ are nonnegative
and remain in $\mathbb{R}_{+}^{8}$, so that the region $\Omega$ is positively invariant.
\end{proof}

\begin{Theorem}
Let
\begin{equation}
\Omega _{h}=\left\{ N_{h} \in \mathbb{R}_{+}:
0\leq N_{h}(t) \leq \frac{\theta _{h}^{\alpha }}{\mu _{h}^{\alpha }}
\right\}  
\label{eq14}
\end{equation}
and
\begin{equation}
\Omega _{r}=\left\{N_{r} \in \mathbb{R}_{+}:
0\leq N_{r}(t)\leq \frac{\theta _{r}^{\alpha }}{\mu _{r}^{\alpha }}\right\}  
\label{eq15}
\end{equation}
with $\Omega$ being defined as
\begin{equation*}
\Omega =\Omega _{h}\times \Omega _{r}.
\end{equation*}
If $N_{h}(0)\leq $ $\frac{\theta _{h}^{\alpha }}{\mu _{h}^{\alpha }}$ and 
$N_{r}(t)\leq \frac{\theta _{r}^{\alpha }}{\mu _{r}^{\alpha }}$, then 
$\Omega$ is the positive invariant region for the model given with 
\eqref{eq11} and \eqref{eq12}.
\end{Theorem}

\begin{proof}
By adding the first five equations side by side in the model \eqref{eq11}, 
we arrive to
\begin{equation}
_{0}^{ABC}D_{t}^{\alpha }N_{h}(t)=\theta _{h}^{\alpha }-\mu _{h}^{\alpha}N_{h}(t)
-\delta _{h}^{\alpha }\left( I_{h}(t)+Q_{h}(t)\right).
\label{eq15a}
\end{equation}
Therefore, Equation \eqref{eq15a} leads to 
\begin{equation*}
_{0}^{ABC}D_{t}^{\alpha }N_{h}\left( t\right) \leq \theta _{h}^{\alpha }-\mu
_{h}^{\alpha }N_{h}\left( t\right) .
\end{equation*}
The Laplace transform of this inequality leads to 
\begin{equation*}
\mathscr{L}\left[ _{0}^{ABC}D_{t}^{\alpha }N_{h}\left( t\right) \right]
\left( s\right) =\frac{\theta _{h}^{\alpha }}{s}-\mu _{h}^{\alpha }
\mathscr{L}\left[ N_{h}\left( t\right) \right] \left( s\right) ,
\end{equation*}
\begin{equation*}
\frac{B(\alpha )}{1-\alpha }\frac{s^{\alpha }\bar{N}_{h}(s)-s^{\alpha
-1}N_{h}(0)}{s^{\alpha }+\frac{\alpha }{1-\alpha }}+\mu _{h}^{\alpha }\leq 
\frac{\theta _{h}^{\alpha }}{s}-\mu _{h}^{\alpha }\bar{N}_{h}(s),
\end{equation*}
where $\bar{N}_{h}\left( s\right) =$ $\left[ N_{h}\left( t\right) \right]
\left( s\right) $ and $N_{h}(0)$ is the total human population at the beginning.
Hence,
\begin{equation*}
\left( \frac{B(\alpha )s^{\alpha }}{\left( 1-\alpha \right) \left( s^{\alpha
}+\frac{\alpha }{1-\alpha }\right) }+\mu _{h}^{\alpha }\right) 
\bar{N}_{h}(s)\leq \theta _{h}^{\alpha }s^{\alpha -\left( \alpha +1\right) }
+\frac{B(\alpha )s^{\alpha -1}N_{h}(0)}{\left( 1-\alpha \right) \left( s^{\alpha }
+\frac{\alpha }{1-\alpha }\right)},
\end{equation*}
\begin{eqnarray*}
\bar{N}_{h}(s) &\leq &\frac{\theta _{h}^{\alpha }\left[ \left( 1-\alpha
\right) s^{\alpha }+\alpha \right] s^{\alpha -\left( \alpha +1\right)
}+B(\alpha )s^{\alpha -1}N_{h}(0)}{B(\alpha )s^{\alpha }+\mu _{h}^{\alpha
}\left( 1-\alpha \right) s^{\alpha }+\alpha \mu _{h}^{\alpha }}, \\
&& \\
\bar{N}_{h}(s) &\leq &\frac{\theta _{h}^{\alpha }\alpha }{\left( 1-\alpha
\right) \mu _{h}^{\alpha }+B(\alpha )}\frac{s^{\alpha -\left( \alpha
+1\right) }}{s^{\alpha }+\frac{\alpha \mu _{h}^{\alpha }}{\left( 1-\alpha
\right) \mu _{h}^{\alpha }+B(\alpha )}} \\
&& \\
&&+\frac{s^{\alpha -1}\left[ \theta _{h}^{\alpha }\left( 1-\alpha \right)
+B(\alpha )N_{h}(0)\right] }{s^{\alpha }\left[ \left( 1-\alpha \right) \mu
_{h}^{\alpha }+B(\alpha )\right] +\alpha \mu _{h}^{\alpha }}.
\end{eqnarray*}
Applying the inverse transform, we obtain
\begin{eqnarray}
N_{h}(t) &\leq &\frac{\theta _{h}^{\alpha }\alpha t^{\alpha }}{\left(
1-\alpha \right) \mu _{h}^{\alpha }+B(\alpha )}E_{\alpha,\alpha +1}\left( 
-\frac{\alpha \mu _{h}^{\alpha }t^{\alpha }}{\left( 1-\alpha \right) \mu
_{h}^{\alpha }+B(\alpha )}\right)  \notag \\
&&+\left[ \left( 1-\alpha \right) \frac{\theta _{h}^{\alpha } +B(\alpha
)N_{h}(0)}{\left( 1-\alpha \right) \mu _{h}^{\alpha }+B(\alpha )}\right]
E_{\alpha,1}\left( -\frac{\alpha \mu _{h}^{\alpha }t^{\alpha }}{\left(
1-\alpha \right) \mu _{h}^{\alpha }+B(\alpha )}\right),  \notag \\
N_{h}(t) &\leq &\frac{\theta _{h}^{\alpha }}{\mu _{h}^{\alpha }},
\label{eq16}
\end{eqnarray}
where the two-parameter Mittag-Leffler function is bounded for all $t>0$ and has
an asymptotic behavior~\cite{Atangana2016}. From~the inequality \eqref{eq16}, 
$N_{h}(t)\leq \frac{\theta _{h}^{\alpha }}{\mu _{h}^{\alpha}}$ 
as $t\rightarrow \infty $. Thence, $N_{h}(t)$ and all other variables of
the monkeypox model $\left( \ref{eq11}\right) $ are bounded in the $\Omega
_{h}$ region. Similar steps are followed by adding the last three equations
in $\left( \ref{eq11}\right) $. Thus, we get $N_{r}(t)\leq \frac{\theta
_{r}^{\alpha }}{\mu _{r}^{\alpha }}$ as $t\rightarrow \infty $.
\end{proof}

\subsection{Existence and~Uniqueness}

Banach fixed point theorem is utilized to show the existence of the solutions
 in the following steps.

Taking the AB fractional integral of both
sides of the model $\left( \ref{eq11}\right) $, we arrive
\begin{eqnarray*}
S_{h}(t)-S_{h}(0) &=&^{AB}I_{0+}^{\alpha }\left\{ \theta _{h}^{\alpha }-
\frac{(\beta _{1}^{\alpha }I_{r}\left( t\right) +\beta _{2}^{\alpha
}I_{h}\left( t\right) )S_{h}\left( t\right) }{N_{h}\left( t\right) }-\mu
_{h}^{\alpha }S_{h}\left( t\right) +\varphi ^{\alpha }Q_{h}\left( t\right)
-u_{1}\left( t\right) S_{h}\left( t\right) \right\}, \\
E_{h}(t)-E_{h}(0) &=&^{AB}I_{0+}^{\alpha }\left\{ \frac{(\beta _{1}^{\alpha
}I_{r}\left( t\right) +\beta _{2}^{\alpha }I_{h}\left( t\right) )S_{h}\left(
t\right) }{N_{h}\left( t\right) }-(\alpha _{1}^{\alpha }+\alpha _{2}^{\alpha
}+\mu _{h}^{\alpha })E_{h}\left( t\right) \right\}, \\
I_{h}(t)-I_{h}(0) &=&^{AB}I_{0+}^{\alpha }\left\{ \alpha _{1}^{\alpha
}E_{h}\left( t\right) -(\mu _{h}^{\alpha }+\delta _{h}^{\alpha }+\gamma
^{\alpha })I_{h}\left( t\right) -u_{2}\left( t\right) I_{h}\left( t\right)
-u_{3}\left( t\right) I_{h}\left( t\right) \right\}, \\
Q_{h}(t)-Q_{h}(0) &=&^{AB}I_{0+}^{\alpha }\left\{ \alpha _{2}^{\alpha
}E_{h}\left( t\right) -(\varphi ^{\alpha }+\tau ^{\alpha }+\mu _{h}^{\alpha
}+\delta _{h}^{\alpha })Q_{h}\left( t\right) +u_{3}\left( t\right)
I_{h}\left( t\right) \right\}, \\
R_{h}(t)-R_{h}(0) &=&^{AB}I_{0+}^{\alpha }\left\{ \gamma ^{\alpha
}I_{h}\left( t\right) +\tau ^{\alpha }Q_{h}\left( t\right) -\mu _{h}^{\alpha
}R_{h}\left( t\right) +u_{1}\left( t\right) S_{h}\left( t\right)
+u_{2}\left( t\right) I_{h}\left( t\right) \right\}, \\
S_{r}(t)-S_{r}(0) &=&^{AB}I_{0+}^{\alpha }\left\{ \theta _{r}^{\alpha }-
\frac{\beta _{3}^{\alpha }S_{r}\left( t\right) I_{r}\left( t\right) }{
N_{r}\left( t\right) }-\mu _{r}^{\alpha }S_{r}\left( t\right) \right\}, 
\end{eqnarray*}
\begin{eqnarray*}
E_{r}(t)-E_{r}(0) &=&^{AB}I_{0+}^{\alpha }\left\{ \frac{\beta _{3}^{\alpha
}S_{r}\left( t\right) I_{r}\left( t\right) }{N_{r}}-(\mu _{r}^{\alpha
}+\alpha _{3}^{\alpha })E_{r}\left( t\right) \right\}, \\
I_{r}(t)-I_{r}(0) &=&^{AB}I_{0+}^{\alpha }\left\{ \alpha _{3}^{\alpha
}E_{r}\left( t\right) -(\mu _{r}^{\alpha }+\delta _{r}^{\alpha })I_{r}\left(
t\right) \right\} .
\end{eqnarray*}
Utilizing definition $\left( \ref{eq9a}\right) $, we have
\begin{eqnarray}
S_{h}(t)-S_{h}(0) &=&\frac{1-\alpha }{B\left( \alpha \right) }G_{1}\left(
t,S_{h}\right) +\frac{\alpha }{B(\alpha )}^{RL}I_{0+}^{\alpha }G_{1}\left(
\theta,S_{h}\right),  \notag \\
E_{h}(t)-E_{h}(0) &=&\frac{1-\alpha }{B\left( \alpha \right) }G_{2}\left(
t,E_{h}\right) +\frac{\alpha }{B(\alpha )}^{RL}I_{0+}^{\alpha }G_{2}\left(
\theta,E_{h}\right),  \notag \\
I_{h}(t)-I_{h}(0) &=&\frac{1-\alpha }{B\left( \alpha \right) }G_{3}\left(
t,I_{h}\right) +\frac{\alpha }{B(\alpha )}^{RL}I_{0+}^{\alpha }G_{3}\left(
\theta,I_{h}\right),  \notag \\
Q_{h}(t)-Q_{h}(0) &=&\frac{1-\alpha }{B\left( \alpha \right) }G_{4}\left(
t,Q_{h}\right) +\frac{\alpha }{B(\alpha )}^{RL}I_{0+}^{\alpha }G_{4}\left(
\theta,Q_{h}\right),  \label{eq17} \\
R_{h}(t)-R_{h}(0) &=&\frac{1-\alpha }{B\left( \alpha \right) }G_{5}\left(
t,R_{h}\right) +\frac{\alpha }{B(\alpha )}^{RL}I_{0+}^{\alpha }G_{5}\left(
\theta,R_{h}\right),  \notag \\
S_{r}(t)-S_{r}(0) &=&\frac{1-\alpha }{B\left( \alpha \right) }G_{6}\left(
t,S_{r}\right) +\frac{\alpha }{B(\alpha )}^{RL}I_{0+}^{\alpha }G_{6}\left(
\theta,S_{r}\right),  \notag \\
E_{r}(t)-E_{r}(0) &=&\frac{1-\alpha }{B\left( \alpha \right) }G_{7}\left(
t,E_{r}\right) +\frac{\alpha }{B(\alpha )}^{RL}I_{0+}^{\alpha }G_{7}\left(
\theta,E_{r}\right),  \notag \\
I_{r}(t)-I_{r}(0) &=&\frac{1-\alpha }{B\left( \alpha \right) }G_{8}\left(
t,I_{r}\right) +\frac{\alpha }{M(\alpha )}^{RL}I_{0+}^{\alpha }G_{8}\left(
\theta,I_{r}\right),  \notag
\end{eqnarray}
where\vspace{-3pt}
\begin{equation}
\left. 
\begin{array}{l}
G_{1}\left( t,S_{h}\right) =\theta _{h}^{\alpha }-\frac{(\beta _{1}^{\alpha
}I_{r}\left( t\right) +\beta _{2}^{\alpha }I_{h}\left( t\right) )S_{h}\left(
t\right) }{N_{h}\left( t\right) }-\mu _{h}^{\alpha }S_{h}\left( t\right)
+\varphi ^{\alpha }Q_{h}\left( t\right) -u_{1}\left( t\right) S_{h}\left(
t\right) , \\ 
G_{2}\left( t,E_{h}\right) =\frac{(\beta _{1}^{\alpha }I_{r}\left( t\right)
+\beta _{2}^{\alpha }I_{h}\left( t\right) )S_{h}\left( t\right) }{
N_{h}\left( t\right) }-(\alpha _{1}^{\alpha }+\alpha _{2}^{\alpha }+\mu
_{h}^{\alpha })E_{h}\left( t\right) , \\ 
G_{3}\left( t,I_{h}\right) =\alpha _{1}^{\alpha }E_{h}\left( t\right) -(\mu
_{h}^{\alpha }+\delta _{h}^{\alpha }+\gamma ^{\alpha })I_{h}\left( t\right)
-u_{2}\left( t\right) I_{h}\left( t\right) -u_{3}\left( t\right) I_{h}\left(
t\right) , \\ 
G_{4}\left( t,Q_{h}\right) =\alpha _{2}^{\alpha }E_{h}\left( t\right)
-(\varphi ^{\alpha }+\tau ^{\alpha }+\mu _{h}^{\alpha }+\delta _{h}^{\alpha
})Q_{h}\left( t\right) +u_{3}\left( t\right) I_{h}\left( t\right) , \\ 
G_{5}\left( t,R_{h}\right) =\gamma ^{\alpha }I_{h}\left( t\right) +\tau
^{\alpha }Q_{h}\left( t\right) -\mu _{h}^{\alpha }R_{h}\left( t\right)
+u_{1}\left( t\right) S_{h}\left( t\right) +u_{2}\left( t\right) I_{h}\left(
t\right) , \\ 
G_{6}\left( t,S_{r}\right) =\theta _{r}^{\alpha }-\frac{\beta _{3}^{\alpha
}S_{r}\left( t\right) I_{r}\left( t\right) }{N_{r}\left( t\right) }-\mu
_{r}^{\alpha }S_{r}\left( t\right) , \\ 
G_{7}\left( t,E_{r}\right) =\frac{\beta _{3}^{\alpha }S_{r}\left( t\right)
I_{r}\left( t\right) }{N_{r}}-(\mu _{r}^{\alpha }+\alpha _{3}^{\alpha
})E_{r}\left( t\right) , \\ 
G_{8}\left( t,I_{r}\right) =\alpha _{3}^{\alpha }E_{r}\left( t\right) -(\mu
_{r}^{\alpha }+\delta _{r}^{\alpha })I_{r}\left( t\right).
\end{array}
\right\}  
\label{eq18}
\end{equation}

We denote by $C(D)$ the Banach space of continuous functions on interval 
$D=\left[ 0,T\right] $ and $F=C(D)\times C(D)\times C(D)\times C(D)\times C(D)\times
C(D)\times C(D)\times C(D)$ with norm $\left\Vert \left(
S_{h},E_{h},I_{h},Q_{h},R_{h},S_{r},E_{r},I_{r}\right) \right\Vert =$ $\
\left\Vert S_{h}\right\Vert +\left\Vert E_{h}\right\Vert +\left\Vert
I_{h}\right\Vert +\left\Vert Q_{h}\right\Vert +\left\Vert R_{h}\right\Vert
+\left\Vert S_{r}\right\Vert +\left\Vert E_{r}\right\Vert +\left\Vert
I_{r}\right\Vert $,~where
\begin{equation}
\left. 
\begin{array}{l}
\left\Vert S_{h}\right\Vert =\sup\limits_{t\in D}\left\vert S_{h}\left(
t\right) \right\vert ,\text{ }\left\Vert E_{h}\right\Vert =\sup\limits_{t\in
D}\left\vert E_{h}\left( t\right) \right\vert ,\text{ }\left\Vert
I_{h}\right\Vert =\sup\limits_{t\in D}\left\vert I_{h}\left( t\right)
\right\vert ,\text{ }\left\Vert Q_{h}\right\Vert =\sup\limits_{t\in
D}\left\vert Q_{h}\left( t\right) \right\vert , \\ 
\left\Vert R_{h}\right\Vert =\sup\limits_{t\in D}\left\vert R_{h}\left(
t\right) \right\vert ,\text{ }\left\Vert S_{r}\right\Vert =\sup\limits_{t\in
D}\left\vert S_{r}\left( t\right) \right\vert ,\text{ }\left\Vert
E_{r}\right\Vert =\sup\limits_{t\in D}\left\vert E_{r}\left( t\right)
\right\vert ,\text{ }\left\Vert I_{r}\right\Vert =\sup\limits_{t\in
D}\left\vert I_{r}\left( t\right) \right\vert.
\end{array}
\right.  
\label{eq19}
\end{equation}

\begin{Theorem}
If the inequality 
\begin{equation*}
0\leq \frac{\beta _{1}^{\alpha }m_{8}+\beta _{2}^{\alpha }m_{3}}{N_{h}\left(
t\right) }+\mu _{h}^{\alpha }+u_{1}\left( t\right) <1
\end{equation*}
holds, then $G_{1}$ is a contraction and satisfies the Lipschitz condition.
\end{Theorem}

\begin{proof}
Let $S_{h}$ and $S_{h1}$ be two functions. Thus, we have
\begin{equation}
\begin{split}	
\Vert G_{1}&\left( t,S_{h}\right) -G_{1}\left( t,S_{h1}\right)
\Vert \\
&=\left\Vert 
\begin{array}{c}
\theta _{h}^{\alpha }-\frac{(\beta _{1}^{\alpha }I_{r}\left( t\right) +\beta
_{2}^{\alpha }I_{h}\left( t\right) )S_{h}\left( t\right) }{N_{h}\left(
t\right) }-\mu _{h}^{\alpha }S_{h}\left( t\right) +\varphi ^{\alpha
}Q_{h}\left( t\right) -u_{1}\left( t\right) S_{h}\left( t\right) \\ 
-\left( \theta _{h}^{\alpha }-\frac{(\beta _{1}^{\alpha }I_{r}\left(
t\right) +\beta _{2}^{\alpha }I_{h}\left( t\right) )S_{h1}\left( t\right) }{
N_{h}\left( t\right) }-\mu _{h}^{\alpha }S_{h1}\left( t\right) +\varphi
^{\alpha }Q_{h}\left( t\right) -u_{1}\left( t\right) S_{h1}\left( t\right)
\right)
\end{array}
\right\Vert  \notag \\
&=\left\Vert \frac{(\beta _{1}^{\alpha }I_{r}\left( t\right) +\beta
_{2}^{\alpha }I_{h}\left( t\right) )}{N_{h}\left( t\right) }\left(
S_{h1}\left( t\right) -S_{h}\left( t\right) \right) +\mu _{h}^{\alpha
}\left( S_{h1}\left( t\right) -S_{h}\left( t\right) \right) +u_{1}\left(
S_{h1}\left( t\right) -S_{h}\left( t\right) \right) \right\Vert  \notag \\
&\leq \left( \frac{\beta _{1}^{\alpha }m_{8}+\beta _{2}^{\alpha }m_{3}}{
N_{h}\left( t\right) }+\mu _{h}^{\alpha }+u_{1}\left( t\right) \right)
\left\Vert S_{h1}\left( t\right) -S_{h}\left( t\right) \right\Vert.  \notag
\end{split}
\end{equation}
Now, let $L_{1}=\frac{\beta _{1}^{\alpha }m_{8}+\beta _{2}^{\alpha }m_{3}}{
N_{h}\left( t\right) }+\mu _{h}^{\alpha }+u_{1}\left( t\right) $, where
$\left\Vert S_{h}\right\Vert \leq m_{1},\left\Vert E_{h}\right\Vert \leq
m_{2},\left\Vert I_{h}\right\Vert \leq m_{3},\linebreak  \left\Vert Q_{h}\right\Vert
\leq m_{4},\left\Vert R_{h}\right\Vert \leq m_{5},\left\Vert
S_{r}\right\Vert \leq m_{6},\left\Vert E_{r}\right\Vert \leq m_{7}$, and~
$\left\Vert I_{r}\right\Vert \leq m_{8}$ are bounded functions. We~obtain
\begin{equation*}
\left\Vert G_{1}\left( t,S_{h}\right) -G_{1}\left( t,S_{h1}\right)
\right\Vert \leq L_{1}\left\Vert S_{h1}\left( t\right) -S_{h}\left( t\right)
\right\Vert .
\end{equation*}
Consequently, the~Lipschitz condition is acquired for kernel $G_{1}$ and 
$0\leq \frac{\beta _{1}^{\alpha }m_{8}+\beta _{2}^{\alpha }m_{3}}{N_{h}\left(
t\right) }+\mu _{h}^{\alpha }+u_{1}\left( t\right) <1$ holds, which
supplies the contraction. Similarly, the~other kernels 
$G_{2}$, $G_{3}$, $G_{4}$, $G_{5}$, $G_{6}$, $G_{7}$ 
and $G_{8}$ supply the Lipschitz
condition and the contraction.
\end{proof}

We can rewrite the kernels in Equation $\left( \ref{eq17}\right) $ as
\vspace{-3pt}
\begin{eqnarray}
S_{h}(t) &=&S_{h}(0)+\frac{1-\alpha }{B\left( \alpha \right) }G_{1}\left(
t,S_{h}\right) +\frac{\alpha }{B(\alpha )}^{RL}I_{0+}^{\alpha }G_{1}\left(
\theta,S_{h}\right),  \notag \\
E_{h}(t) &=&E_{h}(0)+\frac{1-\alpha }{B\left( \alpha \right) }G_{2}\left(
t,E_{h}\right) +\frac{\alpha }{B(\alpha )}^{RL}I_{0+}^{\alpha }G_{2}\left(
\theta,E_{h}\right),  \notag \\
I_{h}(t) &=&I_{h}(0)+\frac{1-\alpha }{B\left( \alpha \right) }G_{3}\left(
t,I_{h}\right) +\frac{\alpha }{B(\alpha )}^{RL}I_{0+}^{\alpha }G_{3}\left(
\theta ,I_{h}\right) ,  \notag \\
Q_{h}(t) &=&Q_{h}(0)+\frac{1-\alpha }{B\left( \alpha \right) }G_{4}\left(
t,Q_{h}\right) +\frac{\alpha }{B(\alpha )}^{RL}I_{0+}^{\alpha }G_{4}\left(
\theta ,Q_{h}\right) ,  \label{eq20} \\
R_{h}(t) &=&R_{h}(0)+\frac{1-\alpha }{B\left( \alpha \right) }G_{5}\left(
t,R_{h}\right) +\frac{\alpha }{B(\alpha )}^{RL}I_{0+}^{\alpha }G_{5}\left(
\theta ,R_{h}\right) ,  \notag \\
S_{r}(t) &=&S_{r}(0)+\frac{1-\alpha }{B\left( \alpha \right) }G_{6}\left(
t,S_{r}\right) +\frac{\alpha }{B(\alpha )}^{RL}I_{0+}^{\alpha }G_{6}\left(
\theta ,S_{r}\right) ,  \notag \\
E_{r}(t) &=&E_{r}(0)+\frac{1-\alpha }{B\left( \alpha \right) }G_{7}\left(
t,E_{r}\right) +\frac{\alpha }{B(\alpha )}^{RL}I_{0+}^{\alpha }G_{7}\left(
\theta ,E_{r}\right) ,  \notag \\
I_{r}(t) &=&I_{r}(0)+\frac{1-\alpha }{B\left( \alpha \right) }G_{8}\left(
t,I_{r}\right) +\frac{\alpha }{B(\alpha )}^{RL}I_{0+}^{\alpha }G_{8}\left(
\theta ,I_{r}\right).  \notag
\end{eqnarray}
Under homogeneous initial conditions and at the time node $t=t_{n}$, we
define the recursive form of $\left( \ref{eq20}\right) $ below:\vspace{-3pt}
\begin{equation}
\begin{array}{cc}
S_{h_{n}}(t)= & \frac{1-\alpha }{B\left( \alpha \right) }G_{1}\left(
t,S_{h_{n-1}}\right) +\frac{\alpha }{B(\alpha )}^{RL}I_{0+}^{\alpha
}G_{1}\left( \theta ,S_{h_{n-1}}\right) , \\ 
E_{h_{n}}(t)= & \frac{1-\alpha }{B\left( \alpha \right) }G_{2}\left(
t,E_{h_{n-1}}\right) +\frac{\alpha }{B(\alpha )}^{RL}I_{0+}^{\alpha
}G_{2}\left( \theta ,E_{h_{n-1}}\right) , \\ 
I_{h_{n}}(t)= & \frac{1-\alpha }{B\left( \alpha \right) }G_{3}\left(
t,I_{h_{n-1}}\right) +\frac{\alpha }{B(\alpha )}^{RL}I_{0+}^{\alpha
}G_{3}\left( \theta ,I_{h_{n-1}}\right) , \\ 
Q_{h_{n}}(t)= & \frac{1-\alpha }{B\left( \alpha \right) }G_{4}\left(
t,Q_{h_{n-1}}\right) +\frac{\alpha }{B(\alpha )}^{RL}I_{0+}^{\alpha
}G_{4}\left( \theta ,Q_{h_{n-1}}\right) , \\ 
R_{h_{n}}(t)= & \frac{1-\alpha }{B\left( \alpha \right) }G_{5}\left(
t,R_{h_{n-1}}\right) +\frac{\alpha }{B(\alpha )}^{RL}I_{0+}^{\alpha
}G_{5}\left( \theta ,R_{h_{n-1}}\right) , \\ 
S_{r_{n}}(t)= & \frac{1-\alpha }{B\left( \alpha \right) }G_{6}\left(
t,S_{r_{n-1}}\right) +\frac{\alpha }{B(\alpha )}^{RL}I_{0+}^{\alpha
}G_{6}\left( \theta ,S_{r_{n-1}}\right) , \\ 
E_{r_{n}}(t)= & \frac{1-\alpha }{B\left( \alpha \right) }G_{7}\left(
t,E_{r_{n-1}}\right) +\frac{\alpha }{B(\alpha )}^{RL}I_{0+}^{\alpha
}G_{7}\left( \theta ,E_{r_{n-1}}\right) , \\ 
I_{r_{n}}(t)= & \frac{1-\alpha }{B\left( \alpha \right) }G_{8}\left(
t,I_{r_{n-1}}\right) +\frac{\alpha }{B(\alpha )}^{RL}I_{0+}^{\alpha
}G_{8}\left( \theta ,I_{r_{n-1}}\right).
\end{array}
\label{eq21}
\end{equation}

Let the differences between consecutive terms in the Equation \eqref{eq21} 
be expressed as
\begin{eqnarray*}
\Psi _{S_{h},n}\left( t\right) &=&S_{h_{n}}\left( t\right)
-S_{h_{n-1}}\left( t\right) =\frac{1-\alpha }{B\left( \alpha \right) }\left(
G_{1}\left( t,S_{h_{n-1}}\right) -G_{1}\left( t,S_{h_{n-2}}\right) \right) \\
&&+\frac{\alpha }{B(\alpha )}^{RL}I_{0+}^{\alpha }\left\{ G_{1}\left( \theta
,S_{h_{n-1}}\right) -G_{1}\left( \theta ,S_{h_{n-2}}\right) \right\} ,
\end{eqnarray*}
\vspace{-12pt}
\begin{eqnarray*}
\Psi _{E_{h},n}\left( t\right) &=&E_{h_{n}}\left( t\right)
-E_{h_{n-1}}\left( t\right) =\frac{1-\alpha }{B\left( \alpha \right) }\left(
G_{2}\left( t,E_{h_{n-1}}\right) -G_{2}\left( t,E_{h_{n-2}}\right) \right) \\
&&+\frac{\alpha }{B(\alpha )}^{RL}I_{0+}^{\alpha }\left\{ G_{2}\left( \theta
,E_{h_{n-1}}\right) -G_{2}\left( \theta ,E_{h_{n-2}}\right) \right\} ,
\end{eqnarray*}
\vspace{-12pt}
\begin{eqnarray*}
\Psi _{I_{h},n}\left( t\right) &=&I_{h_{n}}\left( t\right)
-I_{h_{n-1}}\left( t\right) =\frac{1-\alpha }{B\left( \alpha \right) }\left(
G_{3}\left( t,I_{h_{n-1}}\right) -G_{3}\left( t,I_{h_{n-2}}\right) \right) \\
&&+\frac{\alpha }{B(\alpha )}^{RL}I_{0+}^{\alpha }\left\{ G_{3}\left( \theta
,I_{h_{n-1}}\right) -G_{3}\left( \theta ,I_{h_{n-2}}\right) \right\} ,
\end{eqnarray*}\vspace{-12pt}
\begin{eqnarray*}
\Psi _{Q_{h},n}\left( t\right) &=&Q_{h_{n}}\left( t\right)
-Q_{h_{n-1}}\left( t\right) =\frac{1-\alpha }{B\left( \alpha \right) }\left(
G_{4}\left( t,Q_{h_{n-1}}\right) -G_{4}\left( t,Q_{h_{n-2}}\right) \right) \\
&&+\frac{\alpha }{B(\alpha )}^{RL}I_{0+}^{\alpha }\left\{ G_{4}\left( \theta
,Q_{h_{n-1}}\right) -G_{4}\left( \theta ,Q_{h_{n-2}}\right) \right\} ,
\end{eqnarray*}\vspace{-12pt}
\begin{eqnarray*}
\Psi _{R_{h},n}\left( t\right) &=&R_{h_{n}}\left( t\right)
-R_{h_{n-1}}\left( t\right) =\frac{1-\alpha }{B\left( \alpha \right) }\left(
G_{5}\left( t,R_{h_{n-1}}\right) -G_{5}\left( t,R_{h_{n-2}}\right) \right) \\
&&+\frac{\alpha }{B(\alpha )}^{RL}I_{0+}^{\alpha }\left\{ G_{5}\left( \theta
,R_{h_{n-1}}\right) -G_{5}\left( \theta ,R_{h_{n-2}}\right) \right\} ,
\end{eqnarray*}\vspace{-12pt}
\begin{eqnarray*}
\Psi _{S_{r},n}\left( t\right) &=&S_{r_{n}}\left( t\right)
-S_{r_{n-1}}\left( t\right) =\frac{1-\alpha }{B\left( \alpha \right) }\left(
G_{6}\left( t,S_{r_{n-1}}\right) -G_{6}\left( t,S_{r_{n-2}}\right) \right) \\
&&+\frac{\alpha }{B(\alpha )}^{RL}I_{0+}^{\alpha }\left\{ G_{6}\left( \theta
,S_{r_{n-1}}\right) -G_{6}\left( \theta ,S_{r_{n-2}}\right) \right\} ,
\end{eqnarray*}\vspace{-12pt}
\begin{eqnarray*}
\Psi _{E_{r},n}\left( t\right) &=&E_{r_{n}}\left( t\right)
-E_{r_{n-1}}\left( t\right) =\frac{1-\alpha }{B\left( \alpha \right) }\left(
G_{7}\left( t,E_{r_{n-1}}\right) -G_{7}\left( t,E_{r_{n-2}}\right) \right) \\
&&+\frac{\alpha }{B(\alpha )}^{RL}I_{0+}^{\alpha }\left\{ G_{7}\left( \theta
,E_{r_{n-1}}\right) -G_{7}\left( \theta ,E_{r_{n-2}}\right) \right\} ,
\end{eqnarray*}\vspace{-12pt}
\begin{eqnarray}
\Psi _{I_{r},n}\left( t\right) &=&I_{r_{n}}\left( t\right)
-I_{r_{n-1}}\left( t\right) =\frac{1-\alpha }{B\left( \alpha \right) }
\left\{ G_{8}\left( t,I_{r_{n-1}}\right) -G_{8}\left( t,I_{r_{n-2}}\right)
\right\}  \label{eq22} \\
&&+\frac{\alpha }{B(\alpha )}^{RL}I_{0+}^{\alpha }\left\{ G_{8}\left( \theta
,I_{r_{n-1}}\right) -G_{8}\left( \theta ,I_{r_{n-2}}\right) \right\} . 
\notag
\end{eqnarray}
It is obvious that $S_{h_{n}}\left( t\right) =\overset{n}{\underset{i=0}{
\sum }}\Psi _{S_{h},i}\left( t\right)$, $E_{h_{n}}\left( t\right) =\overset{
n}{\underset{i=0}{\sum }}\Psi _{E_{h},i}\left( t\right) ,$ $I_{h_{n}}\left(
t\right) =\overset{n}{\underset{i=0}{\sum }}\Psi _{I_{h},i}\left( t\right)$,
$Q_{h_{n}}\left( t\right) =\overset{n}{\underset{i=0}{\sum }}\Psi
_{Q_{h},i}\left( t\right) ,$ $R_{h_{n}}\left( t\right) =\overset{n}{\underset
{i=0}{\sum }}\Psi _{R_{h},i}\left( t\right) ,$ $S_{r_{n}}\left( t\right) =
\overset{n}{\underset{i=0}{\sum }}\Psi _{S_{r},i}\left( t\right)$, 
$E_{r_{n}}\left( t\right) =\overset{n}{\underset{i=0}{\sum }}\Psi
_{E_{r},i}\left( t\right) ,$ and $I_{r_{n}}\left( t\right) =\overset{n}{
\underset{i=0}{\sum }}\Psi _{I_{r},i}\left( t\right)$. 
Implementing the norm on Equation \eqref{eq22}, we have 
\begin{equation*}
\begin{split}
\left\Vert \Psi _{S_{h},n}\left( t\right) \right\Vert 
&=\left\Vert S_{h_{n}}\left( t\right) -S_{h_{n-1}}\left( t\right) \right\Vert \\
&\leq \frac{1-\alpha }{B\left( \alpha \right) }\left\Vert \left(
G_{1}\left( t,S_{h_{n-1}}\right) -G_{1}\left( t,S_{h_{n-2}}\right) \right)
\right\Vert \\
& +\frac{\alpha }{B(\alpha )}^{RL}I_{0+}^{\alpha }\left\Vert
\left( G_{1}\left( \theta ,S_{h_{n-1}}\right)
 -G_{1}\left( \theta,S_{h_{n-2}}\right) \right) \right\Vert,
\end{split}
\end{equation*}\vspace{-12pt}
\begin{eqnarray*}
\left\Vert \Psi _{E_{h},n}\left( t\right) \right\Vert &=&\left\Vert
E_{h_{n}}\left( t\right) -E_{h_{n-1}}\left( t\right) \right\Vert \\
&\leq &\frac{1-\alpha }{B\left( \alpha \right) }\left\Vert \left(
G_{2}\left( t,E_{h_{n-1}}\right) -G_{2}\left( t,E_{h_{n-2}}\right) \right)
\right\Vert \\
&& +\frac{\alpha }{B(\alpha )}^{RL}I_{0+}^{\alpha }\left\Vert
\left( G_{2}\left( \theta ,E_{h_{n-1}}\right) -G_{2}\left( \theta
,E_{h_{n-2}}\right) \right) \right\Vert ,
\end{eqnarray*}\vspace{-12pt}
\begin{eqnarray*}
\left\Vert \Psi _{I_{h,}n}\left( t\right) \right\Vert &=&\left\Vert
I_{h_{n}}\left( t\right) -I_{h_{n-1}}\left( t\right) \right\Vert \\
&\leq &\frac{1-\alpha }{B\left( \alpha \right) }\left\Vert \left(
G_{3}\left( t,I_{h_{n-1}}\right) -G_{3}\left( t,I_{h_{n-2}}\right) \right)
\right\Vert \\
&& +\frac{\alpha }{B(\alpha )}^{RL}I_{0+}^{\alpha }\left\Vert
\left( G_{3}\left( \theta ,I_{h_{n-1}}\right) -G_{3}\left( \theta
,I_{h_{n-2}}\right) \right) \right\Vert ,
\end{eqnarray*}\vspace{-12pt}
\begin{eqnarray*}
\left\Vert \Psi _{Q_{h,}n}\left( t\right) \right\Vert &=&\left\Vert
Q_{h_{n}}\left( t\right) -Q_{h_{n-1}}\left( t\right) \right\Vert \\
&\leq &\frac{1-\alpha }{B\left( \alpha \right) }\left\Vert \left(
G_{4}\left( t,Q_{h_{n-1}}\right) -G_{4}\left( t,Q_{h_{n-2}}\right) \right)
\right\Vert \\
&& +\frac{\alpha }{B(\alpha )}^{RL}I_{0+}^{\alpha }\left\Vert
\left( G_{4}\left( \theta ,Q_{h_{n-1}}\right) -G_{4}\left( \theta
,Q_{h_{n-2}}\right) \right) \right\Vert ,
\end{eqnarray*}\vspace{-12pt}
\begin{eqnarray*}
\left\Vert \Psi _{R_{h},n}\left( t\right) \right\Vert &=&\left\Vert
R_{h_{n}}\left( t\right) -R_{h_{n-1}}\left( t\right) \right\Vert \\
&\leq &\frac{1-\alpha }{B\left( \alpha \right) }\left\Vert \left(
G_{5}\left( t,R_{h_{n-1}}\right) -G_{5}\left( t,R_{h_{n-2}}\right) \right)
\right\Vert \\
&&+\frac{\alpha }{B(\alpha )}^{RL}I_{0+}^{\alpha }\left\Vert
\left( G_{5}\left( \theta ,R_{h_{n-1}}\right) -G_{5}\left( \theta
,R_{h_{n-2}}\right) \right) \right\Vert ,
\end{eqnarray*}\vspace{-12pt}
\begin{eqnarray*}
\left\Vert \Psi _{S_{r,}n}\left( t\right) \right\Vert &=&\left\Vert
S_{r_{n}}\left( t\right) -S_{r_{n-1}}\left( t\right) \right\Vert \\
&\leq &\frac{1-\alpha }{B\left( \alpha \right) }\left\Vert \left(
G_{6}\left( t,S_{r_{n-1}}\right) -G_{6}\left( t,S_{r_{n-2}}\right) \right)
\right\Vert \\
&& +\frac{\alpha }{B(\alpha )}^{RL}I_{0+}^{\alpha }\left\Vert
\left( G_{6}\left( \theta ,S_{r_{n-1}}\right) -G_{6}\left( \theta
,S_{r_{n-2}}\right) \right) \right\Vert ,
\end{eqnarray*}\vspace{-12pt}
\begin{eqnarray*}
\left\Vert \Psi _{E_{r},n}\left( t\right) \right\Vert &=&\left\Vert
E_{r_{n}}\left( t\right) -E_{r_{n-1}}\left( t\right) \right\Vert \\
&\leq &\frac{1-\alpha }{B\left( \alpha \right) }\left\Vert \left(
G_{7}\left( t,E_{r_{n-1}}\right) -G_{7}\left( t,E_{r_{n-2}}\right) \right)
\right\Vert \\
&& +\frac{\alpha }{B(\alpha )}^{RL}I_{0+}^{\alpha }\left\Vert
\left( G_{7}\left( \theta ,E_{r_{n-1}}\right) -G_{7}\left( \theta
,E_{r_{n-2}}\right) \right) \right\Vert ,
\end{eqnarray*}\vspace{-12pt}
\begin{eqnarray}
\left\Vert \Psi _{I_{r,}n}\left( t\right) \right\Vert &=&\left\Vert
I_{r_{n}}\left( t\right) -I_{r_{n-1}}\left( t\right) \right\Vert
\label{eq23} \\
&\leq &\frac{1-\alpha }{B\left( \alpha \right) }\left\Vert \left(
G_{8}\left( t,I_{r_{n-1}}\right) -G_{8}\left( t,I_{r_{n-2}}\right) \right)
\right\Vert  \notag \\
&& +\frac{\alpha }{B(\alpha )}^{RL}I_{0+}^{\alpha }\left\Vert
\left( G_{8}\left( \theta ,I_{r_{n-1}}\right) -G_{8}\left( \theta
,I_{r_{n-2}}\right) \right) \right\Vert .  \notag
\end{eqnarray}
Since Lipschitz conditions are satisfied by kernels, 
Equation \eqref{eq23} can be written as
\begin{eqnarray}
\left\Vert \Psi _{S_{h},n}\left( t\right) \right\Vert &=&\left\Vert
S_{h_{n}}\left( t\right) -S_{h_{n-1}}\left( t\right) \right\Vert
\label{eq23a} \\
&\leq &\frac{1-\alpha }{B\left( \alpha \right) }L_{1}\left\Vert
S_{h_{n-1}}-S_{h_{n-2}}\right\Vert +\frac{\alpha }{B(\alpha )}L_{1}\text{ }
^{RL}I_{0+}^{\alpha }\left\Vert S_{h_{n-1}}-S_{h_{n-2}}\right\Vert .  \notag
\end{eqnarray}
Similarly, we get
\begin{equation}
\begin{array}{c}
\left\Vert \Psi _{E_{h},n}\left( t\right) \right\Vert \leq \frac{1-\alpha }{
B\left( \alpha \right) }L_{2}\left\Vert E_{h_{n-1}}-E_{h_{n-2}}\right\Vert +
\frac{\alpha }{B(\alpha )}L_{2}\text{ }^{RL}I_{0+}^{\alpha }\left\Vert
E_{h_{n-1}}-E_{h_{n-2}}\right\Vert , \\ 
\left\Vert \Psi _{I_{h},n}\left( t\right) \right\Vert \leq \frac{1-\alpha }{
B\left( \alpha \right) }L_{3}\left\Vert I_{h_{n-1}}-I_{h_{n-2}}\right\Vert +
\frac{\alpha }{B(\alpha )}L_{3}\text{ }^{RL}I_{0+}^{\alpha }\left\Vert
I_{h_{n-1}}-I_{h_{n-2}}\right\Vert , \\ 
\left\Vert \Psi _{Q_{h},n}\left( t\right) \right\Vert \leq \frac{1-\alpha }{
B\left( \alpha \right) }L_{4}\left\Vert Q_{h_{n-1}}-Q_{h_{n-2}}\right\Vert +
\frac{\alpha }{B(\alpha )}L_{4}\text{ }^{RL}I_{0+}^{\alpha }\left\Vert
Q_{h_{n-1}}-Q_{h_{n-2}}\right\Vert , \\ 
\left\Vert \Psi _{R_{h},n}\left( t\right) \right\Vert \leq \frac{1-\alpha }{
B\left( \alpha \right) }L_{5}\left\Vert R_{h_{n-1}}-R_{h_{n-2}}\right\Vert +
\frac{\alpha }{B(\alpha )}L_{5}\text{ }^{RL}I_{0+}^{\alpha }\left\Vert
R_{h_{n-1}}-R_{h_{n-2}}\right\Vert , \\ 
\left\Vert \Psi _{S_{r},n}\left( t\right) \right\Vert \leq \frac{1-\alpha }{
B\left( \alpha \right) }L_{6}\left\Vert S_{r_{n-1}}-S_{r_{n-2}}\right\Vert +
\frac{\alpha }{B(\alpha )}L_{6}\text{ }^{RL}I_{0+}^{\alpha }\left\Vert
S_{r_{n-1}}-S_{r_{n-2}}\right\Vert ,, \\ 
\left\Vert \Psi _{E_{r},n}\left( t\right) \right\Vert \leq \frac{1-\alpha }{
B\left( \alpha \right) }L_{7}\left\Vert E_{r_{n-1}}-E_{r_{n-2}}\right\Vert +
\frac{\alpha }{B(\alpha )}L_{7}\text{ }^{RL}I_{0+}^{\alpha }\left\Vert
E_{r_{n-1}}-E_{r_{n-2}}\right\Vert , \\ 
\left\Vert \Psi _{I_{r},n}\left( t\right) \right\Vert \leq \frac{1-\alpha }{
B\left( \alpha \right) }L_{8}\left\Vert I_{r_{n-1}}-I_{r_{n-2}}\right\Vert +
\frac{\alpha }{B(\alpha )}L_{8}\text{ }^{RL}I_{0+}^{\alpha }\left\Vert
I_{r_{n-1}}-I_{r_{n-2}}\right\Vert.
\end{array}
\label{eq23b}
\end{equation}

\begin{Theorem}
The model $\left( \ref{eq11}\right) $ has a solution if $M_{0}$ that
satisfies the inequality
\begin{equation}
\left( \frac{1-\alpha }{B\left( \alpha \right) }+\frac{M_{0}^{\alpha }}{
B(\alpha )\Gamma \left( \alpha \right) }\right) L_{i}<1,\text{ }i=1,2,\ldots,8,
\label{eq26}
\end{equation}
can be found.
\end{Theorem}

\begin{proof}
We have shown that functions $S_{h}\left( t\right) ,$ $E_{h}\left( t\right)$, 
$I_{h}\left( t\right) ,$ $Q_{h}\left( t\right) ,$ $R_{h}\left( t\right)$, 
$S_{r}\left( t\right) ,$ $E_{r}\left( t\right) ,$\ and $I_{r}\left( t\right) $
are bounded, and~their kernels satisfy the Lipschitz condition.
Applying \mbox{Equations $\left( \ref{eq23a}\right) $ and $\left( \ref{eq23b}\right) $}
along with a recursive method, we attain
\begin{equation}
\begin{array}{c}
\left\Vert \Psi _{S_{h},n}\left( t\right) \right\Vert \leq \left\Vert
S_{h}\left( 0\right) \right\Vert \left[ \left( \frac{1-\alpha }{B\left(
\alpha \right) }+\frac{M_{0}^{\alpha }}{B(\alpha )\Gamma \left( \alpha
\right) }\right) L_{1}\right] ^{n}, \\ 
\left\Vert \Psi _{E_{h},n}\left( t\right) \right\Vert \leq \left\Vert
E_{h}\left( 0\right) \right\Vert \left[ \left( \frac{1-\alpha }{B\left(
\alpha \right) }+\frac{M_{0}^{\alpha }}{B(\alpha )\Gamma \left( \alpha
\right) }\right) L_{2}\right] ^{n}, \\ 
\left\Vert \Psi _{I_{h},n}\left( t\right) \right\Vert \leq \left\Vert
I_{h}\left( 0\right) \right\Vert \left[ \left( \frac{1-\alpha }{B\left(
\alpha \right) }+\frac{M_{0}^{\alpha }}{B(\alpha )\Gamma \left( \alpha
\right) }\right) L_{3}\right] ^{n}, \\ 
\left\Vert \Psi _{Q_{h},n}\left( t\right) \right\Vert \leq \left\Vert
Q_{h}\left( 0\right) \right\Vert \left[ \left( \frac{1-\alpha }{B\left(
\alpha \right) }+\frac{M_{0}^{\alpha }}{B(\alpha )\Gamma \left( \alpha
\right) }\right) L_{4}\right] ^{n}, \\ 
\left\Vert \Psi _{R_{h},n}\left( t\right) \right\Vert \leq \left\Vert
R_{h}\left( 0\right) \right\Vert \left[ \left( \frac{1-\alpha }{B\left(
\alpha \right) }+\frac{M_{0}^{\alpha }}{B(\alpha )\Gamma \left( \alpha
\right) }\right) L_{5}\right] ^{n}, \\ 
\left\Vert \Psi _{S_{r},n}\left( t\right) \right\Vert \leq \left\Vert
S_{r}\left( 0\right) \right\Vert \left[ \left( \frac{1-\alpha }{B\left(
\alpha \right) }+\frac{M_{0}^{\alpha }}{B(\alpha )\Gamma \left( \alpha
\right) }\right) L_{6}\right] ^{n}, \\ 
\left\Vert \Psi _{E_{r},n}\left( t\right) \right\Vert \leq \left\Vert
E_{r}\left( 0\right) \right\Vert \left[ \left( \frac{1-\alpha }{B\left(
\alpha \right) }+\frac{M_{0}^{\alpha }}{B(\alpha )\Gamma \left( \alpha
\right) }\right) L_{7}\right] ^{n}, \\ 
\left\Vert \Psi _{I_{r},n}\left( t\right) \right\Vert \leq \left\Vert
I_{r}\left( 0\right) \right\Vert \left[ \left( \frac{1-\alpha }{B\left(
\alpha \right) }+\frac{M_{0}^{\alpha }}{B(\alpha )\Gamma \left( \alpha
\right) }\right) L_{8}\right] ^{n}.
\end{array}
\label{eq27}
\end{equation}

We show continuous solutions exist for~the model $\left( 
\ref{eq11}\right) $. To~indicate that the functions $S_{h}\left( t\right)$, 
$E_{h}\left( t\right) ,$ $I_{h}\left( t\right) ,$ $Q_{h}\left( t\right)$, 
$R_{h}\left( t\right) ,$ $S_{r}\left( t\right) ,$ $E_{r}\left( t\right) ,$\
and $I_{r}\left( t\right) $ are solutions of the model \eqref{eq11}, we suppose that
\begin{eqnarray}
S_{h}(t)-S_{h}(0) &=&S_{h,n}\left( t\right) -\varphi _{1n}\left( t\right) , 
\notag \\
E_{h}(t)-E_{h}(0) &=&E_{h,n}\left( t\right) -\varphi _{2n}\left( t\right) , 
\notag \\
I_{h}(t)-I_{h}(0) &=&I_{h,n}\left( t\right) -\varphi _{3n}\left( t\right) , 
\notag \\
Q_{h}(t)-Q_{h}(0) &=&Q_{h,n}\left( t\right) -\varphi _{4n}\left( t\right) ,
\\
R_{h}(t)-R_{h}(0) &=&R_{h,n}\left( t\right) -\varphi _{5n}\left( t\right) , 
\notag \\
S_{r}(t)-S_{r}(0) &=&S_{r,n}\left( t\right) -\varphi _{6n}\left( t\right) , 
\notag \\
E_{r}(t)-E_{r}(0) &=&E_{r,n}\left( t\right) -\varphi _{7n}\left( t\right) , 
\notag \\
I_{r}(t)-I_{r}(0) &=&I_{r,n}\left( t\right) -\varphi _{8n}\left( t\right).
\notag
\end{eqnarray}
Therefore, we achieve
\begin{eqnarray*}
\left\Vert \varphi _{1n}\left( t\right) \right\Vert &=&\left\Vert \frac{
1-\alpha }{B\left( \alpha \right) }\left( G_{1}\left( t,S_{h}\right)
-G_{1}\left( t,S_{h_{n-1}}\right) \right) +\frac{\alpha }{B(\alpha )}
^{RL}I_{0+}^{\alpha }\left\{ G_{1}\left( \theta ,S_{h}\right) -G_{1}\left(
\theta ,S_{h_{n-1}}\right) \right\} \right\Vert \\
&\leq &\frac{1-\alpha }{B\left( \alpha \right) }L_{1}\left\Vert
S_{h}-S_{h_{n-1}}\right\Vert +\frac{\alpha }{B(\alpha )}L_{1}\text{ }
^{RL}I_{0+}^{\alpha }\left\Vert S_{h}-S_{h_{n-1}}\right\Vert  \notag \\
&\leq &\frac{1-\alpha }{B\left( \alpha \right) }L_{1}\left\Vert
S_{h}-S_{h_{n-1}}\right\Vert +\frac{t^{\alpha }}{B(\alpha )\Gamma \left(
\alpha \right) }L_{1}\text{ }\left\Vert S_{h}-S_{h_{n-1}}\right\Vert . 
\notag
\end{eqnarray*}

On employing this process recursively, we obtain at $t=M_{0}$ that
\begin{equation*}
\left\Vert \varphi _{1n}\left( t\right) \right\Vert \leq \left( \frac{
1-\alpha }{B\left( \alpha \right) }+\frac{M_{0}^{\alpha }}{B(\alpha )\Gamma
\left( \alpha \right) }\right) ^{n+1}L_{1}^{n+1}M_{1}.
\end{equation*}
Taking the limit $n\rightarrow \infty $, we obtain $\left\Vert \varphi
_{1n}\left( t\right) \right\Vert \rightarrow 0$. Hence, the~proof is
complete. Similarly, we can prove that $\left\Vert \varphi _{2n}\left(
t\right) \right\Vert \rightarrow 0,$ $\left\Vert \varphi _{3n}\left(
t\right) \right\Vert \rightarrow 0,$ $\left\Vert \varphi _{4n}\left(
t\right) \right\Vert \rightarrow 0,$ $\left\Vert \varphi _{5n}\left(
t\right) \right\Vert \rightarrow 0,$ $\left\Vert \varphi _{6n}\left(
t\right) \right\Vert \rightarrow 0,$ $\left\Vert \varphi _{7n}\left(
t\right) \right\Vert \rightarrow 0$ and $\left\Vert \varphi _{1n}\left(
t\right) \right\Vert \rightarrow 0.$
\end{proof}

The existence of solutions are guaranteed by Banach fixed point theorem. Now,\linebreak  
Theorem~\ref{thm:5} will be given for the uniqueness of the~solution.

\begin{Theorem}
\label{thm:5}
Under the condition that
\begin{equation*}
\left( \frac{1-\alpha }{B\left( \alpha \right) }+\frac{t^{\alpha }}{B(\alpha
)\Gamma \left( \alpha \right) }\right) L_{i}<1,\text{ }i=1,2,\ldots,8,
\end{equation*}
the model \eqref{eq11} has a unique solution, 
\end{Theorem}

\begin{proof}
Suppose $S_{h1}\left( t\right) ,$ $E_{h1}\left( t\right) ,$ $I_{h1}\left(
t\right) ,$ $Q_{h1}\left( t\right) ,$ $R_{h1}\left( t\right)$, 
$S_{r1}\left( t\right) ,$ $E_{r1}\left( t\right) ,$ $I_{r1}\left( t\right) $
are also solutions of $\left( \ref{eq11}\right) $. Then, 
\begin{equation*}
S_{h}(t)-S_{h1}(t)=\frac{1-\alpha }{B\left( \alpha \right) }\left(
G_{1}\left( t,S_{h}\right) -G_{1}\left( t,S_{h1}\right) \right) +\frac{
\alpha }{B(\alpha )}^{RL}I_{0+}^{\alpha }\left( G_{1}\left( \theta
,S_{h}\right) -G_{1}\left( \theta ,S_{h1}\right) \right) .
\end{equation*}
Taking the norm of both sides, we get
\begin{equation*}
\left\Vert S_{h}(t)-S_{h1}(t)\right\Vert \leq \frac{1-\alpha }{B\left(
\alpha \right) }L_{1}\left\Vert S_{h}-S_{h1}\right\Vert +\frac{t^{\alpha }}{
B(\alpha )\Gamma \left( \alpha \right) }L_{1}\left\Vert
S_{h}-S_{h1}\right\Vert .
\end{equation*}
Since $\left( 1-L_{1}\left( \frac{1-\alpha }{B\left( \alpha \right) }+\frac{
t^{\alpha }}{B(\alpha )\Gamma \left( \alpha \right) }\right) \right) >0,$ we
obtain $\left\Vert S_{h}(t)-S_{h1}(t)\right\Vert =0.$ As a result, we get 
$S_{h}(t)=S_{h1}(t).$ Similarly, it is seen that $E_{h}(t)=E_{h1}(t)$, 
$I_{h}(t)=I_{h1}(t),$ $Q_{h}(t)=Q_{h1}(t),$ $R_{h}(t)=R_{h1}(t)$, 
$S_{r}(t)=S_{r1}(t),$ $E_{r}(t)=E_{r1}(t),$ $I_{r}(t)=I_{r1}(t)$ 
and the proof is complete.
\end{proof}


\section{Stability~Analysis}
\label{sec:5}

{Since system~\eqref{eq11} is formulated with 
the Atangana--Baleanu--Caputo fractional derivative $(0<\alpha<1)$,
the stability notion used in this section is \emph{local asymptotic 
stability for fractional-order systems},
investigated via Jacobian linearization 
together with the fractional eigenvalue condition
$|\arg(\lambda_i)|>\alpha\pi/2$.}

\subsection{Equilibrium~Points}

The equilibrium points of the fractional monkeypox model \eqref{eq11}
are found by equating the right-hand side of the system to zero.
The disease-free equilibrium point refers to the absence of disease for the
human and rodent populations in the system. That is, for~$I_{h}=0,I_{r}=0,$
the equilibrium point of the system is calculated as
\begin{equation}
E_{0}=\left( \frac{\theta _{h}^{\alpha }}{\mu _{h}^{\alpha }},0,0,0,0,\frac{
\theta _{r}^{\alpha }}{\mu _{r}^{\alpha }},0,0\right).  \label{eq28}
\end{equation}

In cases where virus spread is seen in populations (i.e., $I_{h}\neq
0,I_{r}\neq 0$), the~endemic equilibrium point is
\begin{equation*}
E_{\ast }=\left( S_{h\ast },E_{h\ast },I_{h\ast },Q_{h\ast },R_{h\ast
},S_{r\ast },E_{r\ast },I_{r\ast }\right)
\end{equation*}
{such that}
\begin{equation}
\left. 
\begin{array}{l}
S_{h\ast }=\frac{k_{1}k_{2}k_{3}\theta _{h}^{\alpha }}{k_{1}k_{2}k_{3}\left(
\phi _{h}+\mu _{h}^{\alpha }+u_{1}\right) -\varphi ^{\alpha }\phi _{h}\left(
\alpha _{2}^{\alpha }k_{2}+\alpha _{1}^{\alpha }u_{3}\right) }, \\ 
E_{h\ast }=\frac{k_{2}k_{3}\phi _{h}\theta _{h}^{\alpha }}{
k_{1}k_{2}k_{3}\left( \phi _{h}+\mu _{h}^{\alpha }+u_{1}\right) -\varphi
^{\alpha }\phi _{h}\left( \alpha _{2}^{\alpha }k_{2}+\alpha _{1}^{\alpha
}u_{3}\right) }, \\ 
I_{h\ast }=\frac{k_{3}\phi _{h}\theta _{h}^{\alpha }\alpha _{1}^{\alpha }}{
k_{1}k_{2}k_{3}\left( \phi _{h}+\mu _{h}^{\alpha }+u_{1}\right) -\varphi
^{\alpha }\phi _{h}\left( \alpha _{2}^{\alpha }k_{2}+\alpha _{1}^{\alpha
}u_{3}\right) }, \\ 
Q_{h\ast }=\frac{\phi _{h}\theta _{h}^{\alpha }\left( \alpha _{2}^{\alpha
}k_{2}+\alpha _{1}^{\alpha }u_{3}\right) }{k_{1}k_{2}k_{3}\left( \phi
_{h}+\mu _{h}^{\alpha }+u_{1}\right) -\varphi ^{\alpha }\phi _{h}\left(
\alpha _{2}^{\alpha }k_{2}+\alpha _{1}^{\alpha }u_{3}\right) }, \\ 
R_{h\ast }=\frac{\theta _{h}^{\alpha }\left( \left( \gamma ^{\alpha
}+u_{2}\right) k_{3}\phi _{h}\alpha _{1}^{\alpha }+k_{1}k_{2}k_{3}u_{1}+\tau
^{\alpha }\phi _{h}\left( \alpha _{2}^{\alpha }k_{2}+\alpha _{1}^{\alpha
}u_{3}\right) \right) }{\mu _{h}^{\alpha }\left( k_{1}k_{2}k_{3}\left( \phi
_{h}+\mu _{h}^{\alpha }+u_{1}\right) -\varphi ^{\alpha }\phi _{h}\left(
\alpha _{2}^{\alpha }k_{2}+\alpha _{1}^{\alpha }u_{3}\right) \right) }, \\ 
S_{r\ast }=\frac{\theta _{r}^{\alpha }}{\mu _{r}^{\alpha }+\phi _{r}}, \\ 
E_{r\ast }=\frac{\theta _{r}^{\alpha }}{k_{4}\left( \mu _{r}^{\alpha }+\phi
_{r}\right) }, \\ 
I_{r\ast }=\frac{\phi _{r}\alpha _{3}^{\alpha }\theta _{r}^{\alpha }}{
k_{4}k_{5}\left( \mu _{r}^{\alpha }+\phi _{r}\right) },
\end{array}
\right\}   \label{eq29}
\end{equation}
where $k_{1}=\alpha _{1}^{\alpha }+\alpha _{2}^{\alpha }+\mu _{h}^{\alpha },$
$k_{2}=\mu _{h}^{\alpha }+\delta _{h}^{\alpha }+\gamma ^{\alpha
}+u_{2}+u_{3},$ $k_{3}=\varphi ^{\alpha }+\tau ^{\alpha }+\mu _{h}^{\alpha
}+\delta _{h}^{\alpha },$ $k_{4}=\mu _{r}^{\alpha }+\alpha _{3}^{\alpha }$, 
$k_{5}=\mu _{r}^{\alpha }+\delta _{r}^{\alpha },$ $\phi _{h}
=\frac{\beta_{1}^{\alpha }I_{r_{\ast }}+\beta _{2}^{\alpha }I_{h_{\ast }}}{N_{h}}$, 
and $\phi _{r}=\frac{\beta _{3}^{\alpha }I_{r_{\ast }}}{N_{r}}.$

\subsection{Basic Reproduction~Number}

The basic reproduction number is the number of expected secondary cases
directly generated by an infected individual in a population in which
individuals are susceptible to the disease. The~$R_{0}$ is a threshold value
used to determine whether the disease will become an epidemic. If~$R_{0}<1$,
monkeypox does not spread in the population, if~$R_{0}>1$, the~disease
spreads to the population and, as a result, some intervention is required to
control the~epidemic.

In~\cite{Peter1}, the~$R_{0}$ is calculated for the integer order
uncontrolled system. For~the controlled system discussed in this article,
the $R_{0}$ should also be recalculated, since the control parameters have
an effect on reducing the transmission of the disease. For~this, new
generation matrix technique~\cite{Diekmann,VanDriessche} is performed: $F$
is the non-negative matrix denoting the rate of occurrence of new infections
in the compartments, $V$ is the non-singular matrix denoting the
transmission rates between compartments, as~follows:
{\begin{equation*}
F=\left[ 
\begin{array}{c}
0 \\ 
\frac{(\beta _{1}^{\alpha }I_{r}\left( t\right) +\beta _{2}^{\alpha
}I_{h}\left( t\right) )S_{h}\left( t\right) }{N_{h}\left( t\right) } \\ 
0 \\ 
0 \\ 
0 \\ 
0 \\ 
0 \\ 
0
\end{array}
\right]
\end{equation*}
and
\begin{equation*}
V=\left[ 
\begin{array}{c}
-\theta _{h}^{\alpha }+\frac{(\beta _{1}^{\alpha }I_{r}\left( t\right)
+\beta _{2}^{\alpha }I_{h}\left( t\right) )S_{h}\left( t\right) }{
N_{h}\left( t\right) }+\mu _{h}^{\alpha }S_{h}\left( t\right) -\varphi
^{\alpha }Q_{h}\left( t\right) +u_{1}\left( t\right) S_{h}\left( t\right) 
\\ 
(\alpha _{1}^{\alpha }+\alpha _{2}^{\alpha }+\mu _{h}^{\alpha })E_{h}\left(
t\right)  \\ 
-\alpha _{1}^{\alpha }E_{h}\left( t\right) +(\mu _{h}^{\alpha }+\delta
_{h}^{\alpha }+\gamma ^{\alpha })I_{h}\left( t\right) +u_{2}\left( t\right)
I_{h}+u_{3}\left( t\right) I_{h}\left( t\right)  \\ 
-\alpha _{2}^{\alpha }E_{h}\left( t\right) +(\varphi ^{\alpha }+\tau
^{\alpha }+\mu _{h}^{\alpha }+\delta _{h}^{\alpha })Q_{h}\left( t\right)
-u_{3}\left( t\right) I_{h}\left( t\right)  \\ 
-\gamma ^{\alpha }I_{h}\left( t\right) -\tau ^{\alpha }Q_{h}\left( t\right)
+\mu _{h}^{\alpha }R_{h}\left( t\right) -u_{1}\left( t\right) S_{h}\left(
t\right) -u_{2}\left( t\right) I_{h}\left( t\right)  \\ 
-\theta _{r}^{\alpha }+\frac{\beta _{3}^{\alpha }S_{r}\left( t\right)
I_{r}\left( t\right) }{N_{r}\left( t\right) }+\mu _{r}^{\alpha }S_{r}\left(
t\right)  \\ 
-\frac{\beta _{3}^{\alpha }S_{r}\left( t\right) I_{r}\left( t\right) }{
N_{r}\left( t\right) }+(\mu _{r}^{\alpha }+\alpha _{3}^{\alpha })E_{r}\left(
t\right)  \\ 
-\alpha _{3}^{\alpha }E_{r}\left( t\right) +(\mu _{r}^{\alpha }+\delta
_{r}^{\alpha })I_{r}\left( t\right) 
\end{array}
\right] .
\end{equation*}}
Hence, $\rho $ denotes the spectral radius of the $FV^{-1}$ matrix. 
After~creating the $F$ and $V$ matrices at the disease-free equilibrium 
point $E_{0}$ as $F_{0}$ and $V_{0}$ matrices, the~largest eigenvalue 
of the multiplication of the matrices $F_{0}$ and $V_{0}^{-1}$ is calculated and
then $R_{0}$ is obtained as: 
{\begin{equation}
R_{0}=\rho \left( F_{0}V_{0}^{-1}\right) =\frac{\alpha _{1}^{\alpha }\beta
_{2}^{\alpha }}{\left( \alpha _{1}^{\alpha }+\alpha _{2}^{\alpha }+\mu
_{h}^{\alpha }\right) \left( \mu _{h}^{\alpha }+\delta _{h}^{\alpha }+\gamma
^{\alpha }+u_{2}\left( t\right) +u_{3}\left( t\right) \right) }.
\label{eq31}
\end{equation}

Note that the spread of monkeypox is reduced by optimal treatment 
and quarantine strategies characterized by the controls 
$u_{2}\left( t\right)$ and $u_{3}\left( t\right)$ as time-dependent functions.}

\subsection{Local Stability~Analysis}

The local stability of the system $\left( \ref{eq11}\right) $ at $E_{0}$ and 
$E_{\ast }$ is~analyzed.

\begin{Theorem}
The controlled system $\left( \ref{eq11}\right) $ at $E_{0}$ is locally
asymptotically stable if the condition
{\begin{equation*}
\left( 1-R_{0}\right) \left( \alpha _{1}^{\alpha }+\alpha _{2}^{\alpha }+\mu
_{h}^{\alpha }\right) \left( \mu _{h}^{\alpha }+\delta _{h}^{\alpha }+\gamma
^{\alpha }+u_{2}\left( t\right)
+u_{3}\left( t\right)\right) >\alpha _{3}^{\alpha }\beta _{3}^{\alpha }
\end{equation*}}
is satisfied for $R_{0}\leq 1.$ Otherwise, the~system is unstable.
\end{Theorem}

\begin{proof}
The Jacobian matrix of system $\left( \ref{eq11}\right) $ at $E_{0}$ is
\begin{equation*}
J_{E_{0}}=\left[ 
\begin{array}{cccccccc}
-\mu _{h}^{\alpha }-{u_{1}\left( t\right)} & 0 & -\beta _{2}^{\alpha } & \varphi ^{\alpha } & 
0 & 0 & 0 & -\beta _{1}^{\alpha } \\ 
0 & -k_{1} & \beta _{2}^{\alpha } & 0 & 0 & 0 & 0 & \beta _{1}^{\alpha } \\ 
0 & \alpha _{1}^{\alpha } & -k_{2} & 0 & 0 & 0 & 0 & 0 \\ 
0 & \alpha _{2}^{\alpha } & {u_{3}\left( t\right)} & -k_{3} & 0 & 0 & 0 & 0 \\ 
u_{1}\left( t\right) & 0 & \gamma ^{\alpha }+{u_{2}\left( t\right)} & \tau ^{\alpha } & -\mu _{h}^{\alpha }
& 0 & 0 & 0 \\ 
0 & 0 & 0 & 0 & 0 & -\mu _{r}^{\alpha } & 0 & -\beta _{3}^{\alpha } \\ 
0 & 0 & 0 & 0 & 0 & 0 & -k_{4} & \beta _{3}^{\alpha } \\ 
0 & 0 & 0 & 0 & 0 & 0 & \alpha _{3}^{\alpha } & -k_{5}
\end{array}
\right].
\end{equation*}
Here, $k_{1}=\alpha _{1}^{\alpha }+\alpha _{2}^{\alpha }+\mu _{h}^{\alpha },$ 
$k_{2}=\mu _{h}^{\alpha }+\delta _{h}^{\alpha }+\gamma ^{\alpha
}+{u_{2}\left( t\right)+u_{3}\left( t\right)},$ $k_{3}=\varphi ^{\alpha }+\tau ^{\alpha }+\mu _{h}^{\alpha
}+\delta _{h}^{\alpha },$\linebreak   $k_{4}=\mu _{r}^{\alpha }+\alpha _{3}^{\alpha }$, 
$k_{5}=\mu _{r}^{\alpha }+\delta _{r}^{\alpha }$. The~eigenvalues of the
matrix $J_{E_{0}}$ are 
\begin{eqnarray*}
\lambda _{1} &=&-\mu _{h}^{\alpha }-{u_{1}\left( t\right)}<0,\text{ }\lambda _{2}=-\left(
\varphi ^{\alpha }+\tau ^{\alpha }+\mu _{h}^{\alpha }+\delta _{h}^{\alpha
}\right) <0,\text{ }\lambda _{3}=-\mu _{h}^{\alpha }<0, \\
\lambda _{4} &=&-\mu _{r}^{\alpha }<0,\text{ }\lambda _{5}=-\left( \mu
_{r}^{\alpha }+\alpha _{3}^{\alpha }\right) <0,\text{ }\lambda _{6}=-\left(
\mu _{r}^{\alpha }+\delta _{r}^{\alpha }\right) <0.
\end{eqnarray*}
The remaining two eigenvalues are calculated by the characteristic equation: 
\begin{equation*}
\lambda ^{2}+A\lambda +B=0,
\end{equation*}
where the coefficients $A$ and $B$ are
\begin{eqnarray*}
A &=&\alpha _{1}^{\alpha }+\alpha _{2}^{\alpha }+\delta _{h}^{\alpha
}+\gamma ^{\alpha }+2\mu _{h}^{\alpha }+{u_{2}\left( t\right)+u_{3}\left( t\right)}, \\
B &=&\left( 1-R_{0}\right) \left( \alpha _{1}^{\alpha }+\alpha _{2}^{\alpha
}+\mu _{h}^{\alpha }\right) \left( \mu _{h}^{\alpha }+\delta _{h}^{\alpha
}+\gamma ^{\alpha }+{u_{2}\left( t\right)+u_{3}\left( t\right)}\right) -\alpha _{3}^{\alpha }\beta
_{3}^{\alpha }.
\end{eqnarray*}

If all eigenvalues of the matrix $J_{E_{0}}$ are negative $\left( \left\vert
\arg \left( \lambda _{i}\right) \right\vert >\frac{\alpha \pi }{2}\right) ,$
the disease-free equilibrium point $E_{0}$ is locally asymptotically stable.
It can be seen from the above that since the coefficient $A$ is always
positive, it can be seen to be locally asymptotic stable. However, the~
conditions for coefficient $B$ are given as: if $R_{0}\leq 1$, 
\begin{equation*}
\left( 1-R_{0}\right) \left( \alpha _{1}^{\alpha }+\alpha _{2}^{\alpha }+\mu
_{h}^{\alpha }\right) \left( \mu _{h}^{\alpha }+\delta _{h}^{\alpha }+\gamma
^{\alpha }+{u_{2}\left( t\right)+u_{3}\left( t\right)}\right) >\alpha _{3}^{\alpha }\beta _{3}^{\alpha }
\end{equation*}
is satisfied. With~this condition fulfilled, the~system is locally
asymptotically stable at $E_{0}$. Otherwise, it is unstable.
\end{proof}

\begin{Theorem}
If $R_{0}>1$, then the~system $\left( \ref{eq11}\right) $ is locally
asymptotically stable at $E_{\ast }.$
\end{Theorem}

\begin{proof}
The Jacobian matrix of system $\left( \ref{eq11}\right) $ at $E_{\ast }$ is
\begin{equation*}
J_{E_{\ast }}=\left[ 
\begin{array}{cccccccc}
-\phi _{h}-\mu _{h}^{\alpha }-{u_{1}\left( t\right)} & 0 & -\frac{\beta _{2}^{\alpha
}S_{h_{\ast }}}{N_{h}\left( t\right)} & \varphi ^{\alpha } & 0 & 0 & 0 & -\frac{\beta
_{1}^{\alpha }S_{h_{\ast }}}{N_{h}\left( t\right)} \\ 
\phi _{h} & -k_{1} & \frac{\beta _{2}^{\alpha }S_{h_{\ast }}}{N_{h}\left( t\right)} & 0 & 0
& 0 & 0 & \frac{\beta _{1}^{\alpha }S_{h_{\ast }}}{N_{h}\left( t\right)} \\ 
0 & \alpha _{1}^{\alpha } & -k_{2} & 0 & 0 & 0 & 0 & 0 \\ 
0 & \alpha _{2}^{\alpha } & {u_{3}\left( t\right)} & -k_{3} & 0 & 0 & 0 & 0 \\ 
u_{1} & 0 & \gamma ^{\alpha }+{u_{2}\left( t\right) }& \tau ^{\alpha } & -\mu _{h}^{\alpha }
& 0 & 0 & 0 \\ 
0 & 0 & 0 & 0 & 0 & -\phi _{r}-\mu _{r}^{\alpha } & 0 & -\frac{\beta
_{3}^{\alpha }S_{r_{\ast }}}{N_{r}\left( t\right)} \\ 
0 & 0 & 0 & 0 & 0 & \phi _{r} & -k_{4} & \frac{\beta _{3}^{\alpha
}S_{r_{\ast }}}{N_{r}\left( t\right)} \\ 
0 & 0 & 0 & 0 & 0 & 0 & \alpha _{3}^{\alpha } & -k_{5}
\end{array}
\right],
\end{equation*}
where $k_{1}=\alpha _{1}^{\alpha }+\alpha _{2}^{\alpha }+\mu _{h}^{\alpha },$
$k_{2}=\mu _{h}^{\alpha }+\delta _{h}^{\alpha }+\gamma ^{\alpha
}+{u_{2}\left( t\right)+u_{3}\left( t\right)},$ $k_{3}=\varphi ^{\alpha }+\tau ^{\alpha }+\mu _{h}^{\alpha
}+\delta _{h}^{\alpha },$ $k_{4}=\mu _{r}^{\alpha }+\alpha _{3}^{\alpha }$, 
$k_{5}=\mu _{r}^{\alpha }+\delta _{r}^{\alpha },$ $\phi _{h}=\frac{\beta
_{1}^{\alpha }I_{r_{\ast }}+\beta _{2}^{\alpha }I_{h_{\ast }}}{N_{h}}$, 
$\phi _{r}=\frac{\beta _{3}^{\alpha }I_{r_{\ast }}}{N_{r}}.$ The eigenvalues
of the matrix $J_{E_{\ast }}$~are
\begin{eqnarray*}
\chi _{1} &=&-\left( \alpha _{1}^{\alpha }+\alpha _{2}^{\alpha }+\mu
_{h}^{\alpha }\right) ,\text{ }\chi _{2}=-\left( \mu _{h}^{\alpha }+\delta
_{h}^{\alpha }+\gamma ^{\alpha }+{u_{2}\left( t\right)+u_{3}\left( t\right)}\right) ,\text{ }\chi _{3}=-\mu
_{h}^{\alpha }, \\
\chi _{4} &=&-\left( \varphi ^{\alpha }+\tau ^{\alpha }+\mu _{h}^{\alpha
}+\delta _{h}^{\alpha }\right) ,\text{ }\chi _{5}=-\left( \mu _{r}^{\alpha
}+\alpha _{3}^{\alpha }\right) ,\text{ }\chi _{6}=-\left( \mu _{r}^{\alpha
}+\delta _{r}^{\alpha }\right) .
\end{eqnarray*}
The characteristic equation giving the other eigenvalues is
\begin{equation*}
\chi ^{2}+C\chi+D=0,
\end{equation*}
where the coefficients $C$ and $D$ are
{\begin{eqnarray*}
C &=&\mu _{h}^{\alpha }+\mu _{r}^{\alpha }+\frac{\beta _{1}^{\alpha
}I_{r_{\ast }}+\beta _{2}^{\alpha }I_{h_{\ast }}}{N_{h}\left( t\right)}+\frac{\beta
_{3}^{\alpha }I_{r_{\ast }}}{N_{r}\left( t\right)}+u_{1}\left( t\right), \\
D &=&\mu _{h}^{\alpha }\mu _{r}^{\alpha }+\left( \frac{\beta _{1}^{\alpha
}I_{r_{\ast }}+\beta _{2}^{\alpha }I_{h_{\ast }}}{N_{h}\left( t\right)}\right) \mu
_{r}^{\alpha }+\frac{\beta _{3}^{\alpha }I_{r_{\ast }}\mu _{h}^{\alpha }}{
N_{r}\left( t\right)}+\frac{\beta _{2}^{\alpha }\beta _{3}^{\alpha }I_{h_{\ast }}I_{r_{\ast
}}}{N_{h}\left( t\right)N_{r}\left( t\right)}+\frac{\beta _{1}^{\alpha }\beta _{3}^{\alpha }I_{r_{\ast
}}^{2}}{N_{h}\left( t\right)N_{r}\left( t\right)} \\
&&-\frac{\alpha _{1}^{\alpha }\beta _{2}^{\alpha }S_{h_{\ast }}}{N_{h}\left( t\right)}
-\frac{\alpha _{3}^{\alpha }\beta _{3}^{\alpha }S_{r_{\ast }}}{N_{r}\left( t\right)}+\mu
_{r}^{\alpha }u_{1}\left( t\right)+\frac{\beta _{3}^{\alpha }I_{r_{\ast }}u_{1}\left( t\right)}{N_{r}\left( t\right)}.
\end{eqnarray*}}

In order to the system $\left( \ref{eq11}\right) $ to be asymptotically
stable at the $E_{\ast }$ equilibrium point, the~eigenvalues should be
negative real numbers. According to the Routh-Hurwitz criterion~\cite{Ahmed}, if the
coefficients $C$ and $D$ are complex, they must have a negative real part.
Thus, if~$C,D>0$ or $C^{2}<4D,$ $C<0$ and $\left\vert \arg \left( \chi
_{i}\right) \right\vert >\frac{\alpha \pi }{2},$ the system is
asymptotically stable.
\end{proof}


\section{Optimal Control of~Monkeypox}
\label{sec:6}

Now, the~optimal control problem for the system $\left( \ref{eq11}\right) $
will be formulated. As~seen in the model dynamics, monkeypox virus, which is
seen in rodents and the bodies of wild animals, spreads between humans by
being transmitted from infected animals to humans. Precautionary controls
are needed to prevent the spread of this epidemic. In~this sense, optimal
control theory serves to provide optimal solutions for possible prevention
strategies. In~this study, the~main aim of optimal control is to minimize
both the cost of the suggested controls and the number of infected humans.
The basis system is equipped with control parameters: $u_{1}(t)$ the rate of
vaccination, $u_{2}(t)$ the rate of treatment, and~$u_{3}(t)$ the quarantine
rate.

The objective function, the~total cost to be minimized, is given as
\begin{equation}
J\underset{\min }{\left( u_{1},u_{2},u_{3}\right) }=\int_{0}^{t_{f}}\left[
I_{h}(t)+E_{h}(t)+\frac{w_{1}}{2}u_{1}^{2}(t)+\frac{w_{2}}{2}u_{2}^{2}(t)
+\frac{w_{3}}{2}u_{3}^{2}(t)\right] dt.  \label{eq33}
\end{equation}
Here, we use a quadratic form to measure control costs. In~addition, 
$w_{1}$, $w_{2}$, $w_{3}$ are positive weighting coefficients that can be chosen to
balance the cost of controls. The~integrand in \eqref{eq33}
is called Lagrangian and corresponds to
\begin{equation}
\mathscr{L}\left( E_{h},I_{h},u_{1},u_{2},u_{3}\right) =I_{h}+E_{h}
+\frac{w_{1}}{2}u_{1}^{2}+\frac{w_{2}}{2}u_{2}^{2}+\frac{w_{3}}{2}
u_{3}^{2}.  
\label{eq34}
\end{equation}

The Lipschitz condition is provided to express the existence of optimal
control functions in the controlled system $\left( \ref{eq11}\right) $. As~a
result, the~existence of controls\ $\left( u_{1},u_{2},u_{3}\right) $ are
inferred~\cite{Lukes,Fleming}.

Now, the~optimality conditions of the system will be~obtained.

\subsection*{Optimality~System}

The necessary optimality conditions depend on the state,
co-state, and~control variables. The~Hamiltonian function is defined in the
following to determine the optimality conditions:\vspace{-3pt}
\begin{eqnarray}
\mathcal{H} &=&\mathscr{L}\left( E_{h}(t),I_{h}(t),u_{1}(t),u_{2}(t),u_{3}(t)\right)
+\sum_{i=1}^{8}\lambda _{i}f_{i}  \label{eq35} \\
&=&I_{h}(t)+E_{h}(t)+\frac{w_{1}}{2}u_{1}^{2}(t)+\frac{w_{2}}{2}u_{2}^{2}(t)
+\frac{w_{3}}{2}u_{3}^{2}(t)  \notag \\
&&+\lambda _{1}\left( t\right) \left( \theta _{h}^{\alpha }-\frac{(\beta
_{1}^{\alpha }I_{r}(t)+\beta _{2}^{\alpha }I_{h}(t))S_{h}(t)}{N_{h}(t)}-\mu
_{h}^{\alpha }S_{h}(t)+\varphi ^{\alpha }Q_{h}(t)-u_{1}\left( t\right)
S_{h}(t)\right)   \notag \\
&&+\lambda _{2}\left( t\right) \left( \frac{(\beta _{1}^{\alpha }I_{r}\left(
t\right) +\beta _{2}^{\alpha }I_{h}\left( t\right) )S_{h}\left( t\right) }{
N_{h}\left( t\right) }-(\alpha _{1}^{\alpha }+\alpha _{2}^{\alpha }+\mu
_{h}^{\alpha })E_{h}\left( t\right) \right)   \notag \\
&&+\lambda _{3}\left( t\right) \left( \alpha _{1}^{\alpha }E_{h}\left(
t\right) -(\mu _{h}^{\alpha }+\delta _{h}^{\alpha }+\gamma ^{\alpha
})I_{h}\left( t\right) -u_{2}\left( t\right) I_{h}\left( t\right)
-u_{3}\left( t\right) I_{h}\left( t\right) \right)   \notag \\
&&+\lambda _{4}\left( t\right) \left( \alpha _{2}^{\alpha }E_{h}\left(
t\right) -(\varphi ^{\alpha }+\tau ^{\alpha }+\mu _{h}^{\alpha }+\delta
_{h}^{\alpha })Q_{h}\left( t\right) +u_{3}\left( t\right) I_{h}\left(
t\right) \right)   \notag \\
&&+\lambda _{5}\left( t\right) \left( \gamma ^{\alpha }I_{h}\left( t\right)
+\tau ^{\alpha }Q_{h}\left( t\right) -\mu _{h}^{\alpha }R_{h}\left( t\right)
+u_{1}\left( t\right) S_{h}\left( t\right) +u_{2}\left( t\right) I_{h}\left(
t\right) \right)   \notag \\
&&+\lambda _{6}\left( t\right) \left( \theta _{r}^{\alpha }-\frac{\beta
_{3}^{\alpha }S_{r}\left( t\right) I_{r}\left( t\right) }{N_{r}\left(
t\right) }-\mu _{r}^{\alpha }S_{r}\left( t\right) \right)   \notag \\
&&+\lambda _{7}\left( t\right) \left( \frac{\beta _{3}^{\alpha }S_{r}\left(
t\right) I_{r}\left( t\right) }{N_{r}\left( t\right) }-(\mu _{r}^{\alpha
}+\alpha _{3}^{\alpha })E_{r}\left( t\right) \right)   \notag \\
&&+\lambda _{8}\left( t\right) \left( \alpha _{3}^{\alpha }E_{r}\left(
t\right) -(\mu _{r}^{\alpha }+\delta _{r}^{\alpha })I_{r}\left( t\right)
\right) .  \notag
\end{eqnarray}
Here, $\lambda_{i}$, $i=1,2,\ldots,8$, denote co-state functions 
and $f_{i}$, $i=1,2,\ldots,8$, represent the right-hand 
side of the system \eqref{eq11}.

\begin{Theorem}\label{Theorem8}
Suppose $\left( S_{h}^{\ast },E_{h}^{\ast },I_{h}^{\ast },Q_{h}^{\ast
},R_{h}^{\ast },S_{r}^{\ast },E_{r}^{\ast },I_{r}^{\ast }\right) $ are the
optimal solutions of the system $\left( \ref{eq11}\right) $, and~
$u_{1}^{\ast },u_{2}^{\ast },$ and $u_{3}^{\ast }$ are the optimal controls
minimizing the objective function. Then there are co-state variables 
$\lambda _{i}$ that satisfy the co-state equations
\begin{eqnarray*}
&&_{t}^{ABC}D_{t_{f}}^{\alpha }\lambda _{1}\left( t\right) =-\lambda
_{1}\left( t\right) \left( \frac{(\beta _{1}^{\alpha }I_{r}^{\ast }(t)+\beta
_{2}^{\alpha }I_{h}^{\ast }(t))\left( E_{h}^{\ast }\left( t\right)
+I_{h}^{\ast }(t)+Q_{h}^{\ast }\left( t\right) +R_{h}^{\ast }\left( t\right)
\right) }{\left( N_{h}^{\ast }\left( t\right) \right) ^{2}}+\mu _{h}^{\alpha
}+u_{1}^{\ast }\left( t\right) \right)  \\
&&\text{ \ \ \ \ \ \ \ \ \ \ \ \ \ \ \ \ \ \ }+\lambda _{2}\left( t\right)
\left( \frac{(\beta _{1}^{\alpha }I_{r}^{\ast }(t)+\beta _{2}^{\alpha
}I_{h}^{\ast }(t))\left( E_{h}^{\ast }\left( t\right) +I_{h}^{\ast
}(t)+Q_{h}^{\ast }\left( t\right) +R_{h}^{\ast }\left( t\right) \right) }{
\left( N_{h}^{\ast }\left( t\right) \right) ^{2}}\right) +\lambda _{5}\left(
t\right) u_{1}^{\ast }\left( t)\right)  \\
&&_{t}^{ABC}D_{t_{f}}^{\alpha }\lambda _{2}\left( t\right) =1-\lambda
_{1}\left( t\right) \left( \frac{(\beta _{1}^{\alpha }I_{r}^{\ast }(t)+\beta
_{2}^{\alpha }I_{h}^{\ast }(t))S_{h}^{\ast }(t)}{\left( N_{h}^{\ast }\left(
t\right) \right) ^{2}}\right)  \\
&&\text{ \ \ \ \ \ \ \ \ \ \ \ \ \ \ \ \ \ }-\lambda _{2}\left( t\right)
\left( \frac{(\beta _{1}^{\alpha }I_{r}^{\ast }(t)+\beta _{2}^{\alpha
}I_{h}^{\ast }(t))S_{h}^{\ast }(t)}{\left( N_{h}^{\ast }\left( t\right)
\right) ^{2}}+\alpha _{1}^{\alpha }+\alpha _{2}^{\alpha }+\mu _{h}^{\alpha
}\right) +\lambda _{3}\left( t\right) \alpha _{1}^{\alpha }+\lambda
_{4}\left( t\right) \alpha _{2}^{\alpha } \\
&&_{t}^{ABC}D_{t_{f}}^{\alpha }\lambda _{3}\left( t\right) =1-\lambda
_{1}\left( t\right) \left( \frac{S_{h}^{\ast }(t)\left( \beta _{2}^{\alpha
}\left( S_{h}^{\ast }(t)+E_{h}^{\ast }\left( t\right) +Q_{h}^{\ast }\left(
t\right) +R_{h}^{\ast }\left( t\right) \right) -\beta _{1}^{\alpha
}I_{r}^{\ast }(t)\right) }{\left( N_{h}^{\ast }\left( t\right) \right) ^{2}}
\right)  \\
&&\text{ \ \ \ \ \ \ \ \ \ \ \ \ \ \ \ \ \ \ \ }+\lambda _{2}\left( t\right)
\left( \frac{S_{h}^{\ast }(t)\left( \beta _{2}^{\alpha }\left( S_{h}^{\ast
}(t)+E_{h}^{\ast }\left( t\right) +Q_{h}^{\ast }\left( t\right) +R_{h}^{\ast
}\left( t\right) \right) -\beta _{1}^{\alpha }I_{r}^{\ast }(t)\right) }{
\left( N_{h}^{\ast }\left( t\right) \right) ^{2}}\right)  \\
&&\text{ \ \ \ \ \ \ \ \ \ \ \ \ \ \ \ \ \ \ \ }-\lambda _{3}\left( t\right)
\left( \mu _{h}^{\alpha }+\delta _{h}^{\alpha }+\gamma ^{\alpha
}+u_{2}^{\ast }\left( t\right) +u_{3}^{\ast }\left( t\right) \right)
+\lambda _{4}\left( t\right) u_{3}^{\ast }\left( t\right) +\lambda
_{5}\left( t\right) \left( \gamma ^{\alpha }+u_{2}^{\ast }(t)\right) 
\end{eqnarray*}
\begin{eqnarray*}
&&_{t}^{ABC}D_{t_{f}}^{\alpha }\lambda _{4}\left( t\right) =\lambda
_{1}\left( t\right) \left( \frac{(\beta _{1}^{\alpha }I_{r}^{\ast }(t)+\beta
_{2}^{\alpha }I_{h}^{\ast }(t))S_{h}^{\ast }(t)}{\left( N_{h}^{\ast }\left(
t\right) \right) ^{2}}+\varphi ^{\alpha }\right) -\lambda _{2}\left(
t\right) \left( \frac{(\beta _{1}^{\alpha }I_{r}^{\ast }(t)+\beta
_{2}^{\alpha }I_{h}^{\ast }(t))S_{h}^{\ast }(t)}{\left( N_{h}^{\ast }\left(
t\right) \right) ^{2}}\right)  \\
&&\text{ \ \ \ \ \ \ \ \ \ \ \ \ \ \ \ \ \ \ }-\lambda _{4}\left( t\right)
\left( \varphi ^{\alpha }+\tau ^{\alpha }+\mu _{h}^{\alpha }+\delta
_{h}^{\alpha }\right) +\lambda _{5}\left( t\right) \tau ^{\alpha } \\
&&_{t}^{ABC}D_{t_{f}}^{\alpha }\lambda _{5}\left( t\right) =\lambda
_{1}\left( t\right) \left( \frac{(\beta _{1}^{\alpha }I_{r}^{\ast }(t)+\beta
_{2}^{\alpha }I_{h}^{\ast }(t))S_{h}^{\ast }(t)}{\left( N_{h}^{\ast }\left(
t\right) \right) ^{2}}\right) -\lambda _{2}\left( t\right) \left( \frac{
(\beta _{1}^{\alpha }I_{r}^{\ast }(t)+\beta _{2}^{\alpha }I_{h}^{\ast
}(t))S_{h}^{\ast }(t)}{\left( N_{h}^{\ast }\left( t\right) \right) ^{2}}
\right)  \\
&&\text{ \ \ \ \ \ \ \ \ \ \ \ \ \ \ \ \ \ \ }-\lambda _{5}\left( t\right)
\mu _{h}^{\alpha } \\
&&_{t}^{ABC}D_{t_{f}}^{\alpha }\lambda _{6}\left( t\right) =-\lambda
_{6}\left( t\right) \left( \frac{\beta _{3}^{\alpha }I_{r}^{\ast }\left(
t\right) \left( E_{r}^{\ast }\left( t\right) +I_{r}^{\ast }\left( t\right)
\right) }{\left( N_{r}^{\ast }\left( t\right) \right) ^{2}}+\mu _{r}^{\alpha
}\right) +\lambda _{7}\left( t\right) \left( \frac{\beta _{3}^{\alpha
}I_{r}^{\ast }\left( t\right) \left( E_{r}^{\ast }\left( t\right)
+I_{r}^{\ast }\left( t\right) \right) }{\left( N_{r}^{\ast }\left( t\right)
\right) ^{2}}\right)  \\
&&_{t}^{ABC}D_{t_{f}}^{\alpha }\lambda _{7}\left( t\right) =\lambda
_{6}\left( t\right) \left( \frac{\beta _{3}^{\alpha }S_{r}^{\ast }\left(
t\right) I_{r}^{\ast }\left( t\right) }{\left( N_{r}^{\ast }\left( t\right)
\right) ^{2}}\right) -\lambda _{7}\left( t\right) \left( \frac{\beta
_{3}^{\alpha }S_{r}^{\ast }\left( t\right) I_{r}^{\ast }\left( t\right) }{
\left( N_{r}^{\ast }\left( t\right) \right) ^{2}}+\mu _{r}^{\alpha }+\alpha
_{3}^{\alpha }\right) +\lambda _{8}\left( t\right) \alpha _{3}^{\alpha } \\
&&_{t}^{ABC}D_{t_{f}}^{\alpha }\lambda _{8}\left( t\right) =-\lambda
_{1}\left( t\right) \left( \frac{\beta _{1}^{\alpha }S_{h}^{\ast }(t)}{
N_{h}^{\ast }\left( t\right) }\right) +\lambda _{2}\left( t\right) \left( 
\frac{\beta _{1}^{\alpha }S_{h}^{\ast }(t)}{N_{h}^{\ast }\left( t\right)}
\right) -\lambda _{6}\left( t\right) \left( \frac{\beta _{3}^{\alpha
}S_{r}^{\ast }\left( t\right) \left( S_{r}^{\ast }\left( t\right)
+E_{r}^{\ast }\left( t\right) \right) }{\left( N_{r}^{\ast }\left( t\right)
\right) ^{2}}\right)  \\
&&\text{ \ \ \ \ \ \ \ \ \ \ \ \ \ \ \ \ \ \ }+\lambda _{7}\left( t\right)
\left( \frac{\beta _{3}^{\alpha }S_{r}^{\ast }\left( t\right) \left(
S_{r}^{\ast }\left( t\right) +E_{r}^{\ast }\left( t\right) \right) }{\left(
N_{r}^{\ast }\left( t\right) \right) ^{2}}\right) -\lambda _{8}\left(
t\right) \left( \mu _{r}^{\alpha }+\delta _{r}^{\alpha }\right) 
\end{eqnarray*}
with the transversality conditions
\begin{equation*}
\lambda _{i}\left( t_{f}\right) =0,\text{ }i=1,\ldots,8.
\end{equation*}
Additionally, the~optimal control functions $u_{1}^{\ast },u_{2}^{\ast
},u_{3}^{\ast }$ are given as follows:
\begin{equation}
\left. 
\begin{array}{c}
u_{1}^{\ast }\left( t\right) =\max \left\{ \min \left\{ \frac{\left( \lambda
_{1}(t)+\lambda _{5}(t)\right) S_{h}^{\ast }(t)}{w_{1}},0.9\right\}
,0\right\} , \\ 
u_{2}^{\ast }\left( t\right) =\max \left\{ \min \left\{ \frac{\left( \lambda
_{3}(t)+\lambda _{5}(t)\right) I_{h}^{\ast }(t)}{w_{2}},0.9\right\}
,0\right\} , \\ 
u_{3}^{\ast }\left( t\right) =\max \left\{ \min \left\{ \frac{\left( \lambda
_{3}(t)+\lambda _{4}(t)\right) I_{h}^{\ast }(t)}{w_{3}},0.9\right\}
,0\right\} .
\end{array}
\right. 
\label{bbie3}
\end{equation}
\end{Theorem}

\begin{proof}
The following necessary conditions clearly give the optimality system
consisting of the optimal solutions:
\begin{equation}
\left. 
\begin{array}{c}
_{0}^{ABC}D_{t}^{\alpha }S_{h}=\frac{\partial \mathcal{H}}{\partial \lambda
_{1}},\text{ \ }_{0}^{ABC}D_{t}^{\alpha }E_{h}=\frac{\partial \mathcal{H}}{
\partial \lambda _{2}}, \\ 
_{0}^{ABC}D_{t}^{\alpha }I_{h}=\frac{\partial \mathcal{H}}{\partial \lambda
_{3}},\text{ \ }_{0}^{ABC}D_{t}^{\alpha }Q_{h}=\frac{\partial \mathcal{H}}{
\partial \lambda _{4}}, \\ 
_{0}^{ABC}D_{t}^{\alpha }R_{h}=\frac{\partial \mathcal{H}}{\partial \lambda
_{5}},\text{ \ }_{0}^{ABC}D_{t}^{\alpha }S_{r}=\frac{\partial \mathcal{H}}{
\partial \lambda _{6}}, \\ 
_{0}^{ABC}D_{t}^{\alpha }E_{r}=\frac{\partial \mathcal{H}}{\partial \lambda
_{7}},\text{ \ }_{0}^{ABC}D_{t}^{\alpha }I_{r}=\frac{\partial \mathcal{H}}{
\partial \lambda_{8}}.
\end{array}
\right\}  
\label{eq36}
\end{equation}
\begin{equation}
\left. 
\begin{array}{c}
_{t}^{ABC}D_{t_{f}}^{\alpha }\lambda _{1}=\frac{\partial \mathcal{H}}{
\partial S_{h}},\text{ \ }_{t}^{ABC}D_{t_{f}}^{\alpha }\lambda _{2}
=\frac{\partial \mathcal{H}}{\partial E_{h}}, \\ 
_{t}^{ABC}D_{t_{f}}^{\alpha }\lambda _{3}=\frac{\partial \mathcal{H}}{
\partial I_{h}},\text{ \ }_{t}^{ABC}D_{t_{f}}^{\alpha }\lambda _{4}
=\frac{\partial \mathcal{H}}{\partial Q_{h}}, \\ 
_{t}^{ABC}D_{t_{f}}^{\alpha }\lambda _{5}=\frac{\partial \mathcal{H}}{
\partial R_{h}},\text{ \ }_{t}^{ABC}D_{t_{f}}^{\alpha }\lambda _{6}
=\frac{\partial \mathcal{H}}{\partial S_{r}}, \\ 
_{t}^{ABC}D_{t_{f}}^{\alpha }\lambda _{7}=\frac{\partial \mathcal{H}}{
\partial E_{r}},\text{ \ }_{t}^{ABC}D_{t_{f}}^{\alpha }\lambda _{8}
=\frac{\partial \mathcal{H}}{\partial I_{r}}.
\end{array}
\right\}  
\label{eq37}
\end{equation}
\begin{equation}
\frac{\partial \mathcal{H}}{\partial u_{1}}=0,\text{ }\frac{\partial 
\mathcal{H}}{\partial u_{2}}=0,\text{ }\frac{\partial \mathcal{H}}{\partial
u_{3}}=0.  
\label{eq38}
\end{equation}
Here, Equation $\left( \ref{eq36}\right) $ is the state system, Equation 
\eqref{eq37} is the co-state system corresponding to the result given
in the Theorem \ref{Theorem8}, while Equation $\left( \ref{eq38}\right) $
allow us to derive the controls \eqref{bbie3}.
Also, the~initial conditions are as in Equation $\left( \ref{eq12}\right)$ 
and transversality conditions are given as
\begin{equation*}
\lambda _{i}\left( t_{f}\right) =0,\text{ } i=1,\ldots,8.
\end{equation*}
The proof is complete.
\end{proof}

Next, a~numerical method will be applied to solve the optimality system
and graphical results will be discussed in~detail.


\section{Numerical Results and~Discussion}
\label{sec:7}

In this section, strategies are detailed and numerically simulated to
demonstrate the effects of optimal controls on the fractional monkeypox
model. Adams type predictor--corrector algorithm, adjusted for AB derivatives,
is used to achieve the numerical \mbox{results~\cite{Diethelm,BaleanuJa}.} 
{Let us denote the state vector with}
$\textbf{X}=(S_h,E_h,I_h,Q_h,R_h,S_r,E_r,I_r)$, the~control vector 
with $\textbf{U}=(u_1,u_2,u_3)$, and the costate vector with
$\mathbf{\Lambda}=(\lambda_1,\lambda_2,\lambda_3,
\lambda_4,\lambda_5,\lambda_6,\lambda_7,\lambda_8)$.
Also, the~time interval $[0,t_f]$ is discretized for final 
time $t_f = 36$ months and fixed step size $h = 0.1$. Thus, 
the~discrete state equation on the node $n$, 
where $0\leq n \leq N = t_f/h$, is given by
\begin{equation}
\left.
{\begin{array}{c}
\textbf{X}\left( t_{n+1}\right) =\textbf{X}\left( 0\right) +\frac{h^{\alpha }}{\Gamma \left(
\alpha +2\right) }\left[ \mathbf{G}\left(t_{n+1},\textbf{X}^{p}\left( t_{n+1}\right),
\textbf{U}\left( t_{n+1}\right)\right) \right.\\
+\left. \underset{j=0}{\overset{n}{\sum }}a_{j,n+1}\text{ }\mathbf{G}\left(
t_{j},\textbf{X}\left( t_{j}\right),\textbf{U}\left( t_{j}\right) \right) \right], 
\\
\textbf{X}^{p}\left( t_{n+1}\right) =\textbf{X}\left( 0\right) +\frac{h^{\alpha }}{\Gamma \left(
\alpha +1\right) }\left[\underset{j=0}{\overset{n}{\sum }}b_{j,n+1}\text{ }\mathbf{G}\left(
t_{j},\textbf{X}\left( t_{j}\right),\textbf{U}\left( t_{j}\right) \right) \right],
\end{array}}
\label{bbie1}
\right\}
\end{equation}
{in which $\mathbf{G} = \left(G_1,G_2,G_3,G_4,G_5,G_6,G_7,G_8\right)$ is given 
via Equation$\left(\ref{eq18}\right)$. Then, 
using the forward--backward sweep algorithm, 
the~discrete costate equation}
{\begin{equation*}
_{t}^{ABC}D_{t_f}^{\alpha }\mathbf{\Lambda} \left( t\right) =_{0}^{ABC}D_{t}^{\alpha
}\mathbf{\Lambda} \left( t_f-t\right)
\end{equation*}}
is given as 
\vspace{-3pt}
\begin{equation}
\left.
\begin{array}{c}
\mathbf{\Lambda} \left( t_{N-n-1}\right) =\frac{h^{\alpha }}{\Gamma \left(
\alpha +2\right) }\left[ \frac{\partial
\mathcal{H}}{\partial \textbf{X}}\left( t_{N-n-1},\textbf{X}
\left( t_{N-n-1}\right),\textbf{U}\left( t_{N-n-1}\right)
,\mathbf{\Lambda}^{p}\left( t_{N-n-1}\right) \right) \right. \\
+\left. \underset{j=0}{\overset{n}{\sum }}a_{j,n+1}\frac{\partial 
\mathcal{H}}{\partial X}\left( t_{N-j},\textbf{X}\left( t_{N-j}\right),
\textbf{U}\left( t_{N-j}\right) ,\mathbf{\Lambda}\left(
t_{N-j}\right) \right) \right], 
\\
\mathbf{\Lambda}^{p} \left( t_{N-n-1}\right) =\frac{h^{\alpha }}{\Gamma \left(
\alpha +1\right) }\left[ \underset{j=0}{\overset{n}{\sum }}b_{j,n+1}
\frac{\partial \mathcal{H}}{\partial 
\textbf{X}}\left( t_{N-j},\textbf{X}
\left( t_{N-j}\right),\textbf{U}\left( t_{N-j}\right) ,\mathbf{\Lambda}\left(
t_{N-j}\right) \right) \right].
\end{array}
\right\}
\label{bbie2}
\end{equation}

{The coefficients are as follows:}
\begin{eqnarray*}
a_{j,n+1} &=&\left\{ 
\begin{array}{l}
n^{\alpha +1}-\left( n-\alpha \right) \left( n+1\right) ^{\alpha },\text{ \
\ \ \ \ \ \ \ \ \ \ \ \ \ \ \ \ \ \ \ \ \ \ \ \ \ \ \ \ \ \ \ \ \ \ if \ }j=0
\\ 
\left( n-j+2\right) ^{\alpha +1}-\left( n-j\right) ^{\alpha +1}-2\left(
n-j+1\right) ^{\alpha +1}\text{\ \ if }1\leq \text{\ }j\leq n,
\end{array}
\right.\\
b_{j,n+1} &=&\left( n+1-j\right) ^{\alpha }-\left( n-j\right) ^{\alpha }.
\end{eqnarray*}

All numerical calculations were performed with MATLAB 2021b. In~these numerical
calculations, the~parameters are considered as $\theta _{h}=0.029,$ $\beta
_{1}=0.00025,$ $\beta _{2}=9;$ $\alpha _{1}=0.3,$ $\alpha _{2}=2,$ $\varphi
=2,$ $\tau =0.52,$ $\gamma =(1/21),$ $\mu _{h}=0.02,$ $\delta _{h}=0.2$, 
$\theta _{r}=0.2,$ $\beta _{3}=6,$ $\mu _{r}=1.5,$ $\alpha _{3}=0.2,$ $\delta
_{r}=0.5$ with initial conditions $S_{h}(0)=0.8,$ $E_{h}(0)=0.1$, 
$I_{h}(0)=0.1,$ $Q_{h}(0)=0,$ $R_{h}(0)=0,$ $S_{r}(0)=0.8,$ $E_{r}(0)=0.15$, 
$I_{r}(0)=0.05$. The~numerical scheme utilized is outlined in the following Algorithm~\ref{algorithm1}:

\vspace{12pt}
\begin{algorithm}[H]
\caption{ }\label{algorithm1}
Initiate state and costate vectors. \\
Set the initial guess for the control vector. \\
Calculate predictor and corrector of state vector by Equation $\left(\ref{bbie1}\right)$.\\
Calculate predictor and corrector of costate vector by Equation $\left(\ref{bbie2}\right)$.\\
Update control vector by Equation $\left(\ref{bbie3}\right)$.\\
If the convergence criteria are not reached go to step 3. \\
Optimal state and control vectors are achieved. 
\end{algorithm}
\vspace{12pt}

Our main purpose in these simulations is to examine the effects of
vaccination, treatment, and~quarantine controls on the dynamics against the
spread of the monkeypox virus. Strategy 1 corresponds to the comparison of the
single control effect and the uncontrolled situation. Strategy 2 reveals the
effects of double control on the system. Strategy 3 includes a comparison of
triple control and optimal double control~strategies.

\subsection{Strategy~1}

In this strategy, the~behavior of the single controls on the system is
presented in Figure~\ref{Fig1} with the following cases:

\begin{itemize}
\item Strategy 1.1 applying only vaccination control to susceptible humans,
that is, $u_{1}\neq 0$, $u_{2}=0$, and~$u_{3}=0$;

\item Strategy 1.2 applying only treatment control to infected humans,
that is, $u_{1}=0$, $u_{2}\neq 0$, and~$u_{3}=0$;

\item Strategy 1.3 applying only quarantine control to infected humans,
that is, $u_{1}=0$, $u_{2}=0$, and~$u_{3}\neq 0$.
\end{itemize}

As seen in Figure~\ref{Fig1}, it is clear that the effect of
each control separately is quite effective when compared to the uncontrolled
system. Although~only-treatment and only-quarantine strategies seem to
produce better results for infected people, the~effect of vaccination is
actually too good to be ignored. Because~vaccination control aims to
establish permanent immunity in the population, its effect cannot be
expected to be seen immediately in the infected compartment. In~this
context, while vaccination is effective in the long-term in the infected
compartment, it increases the recovered class by reducing the susceptible and
exposed compartments in a short~time.

\begin{figure}[ht]
\includegraphics[width=.9\linewidth]{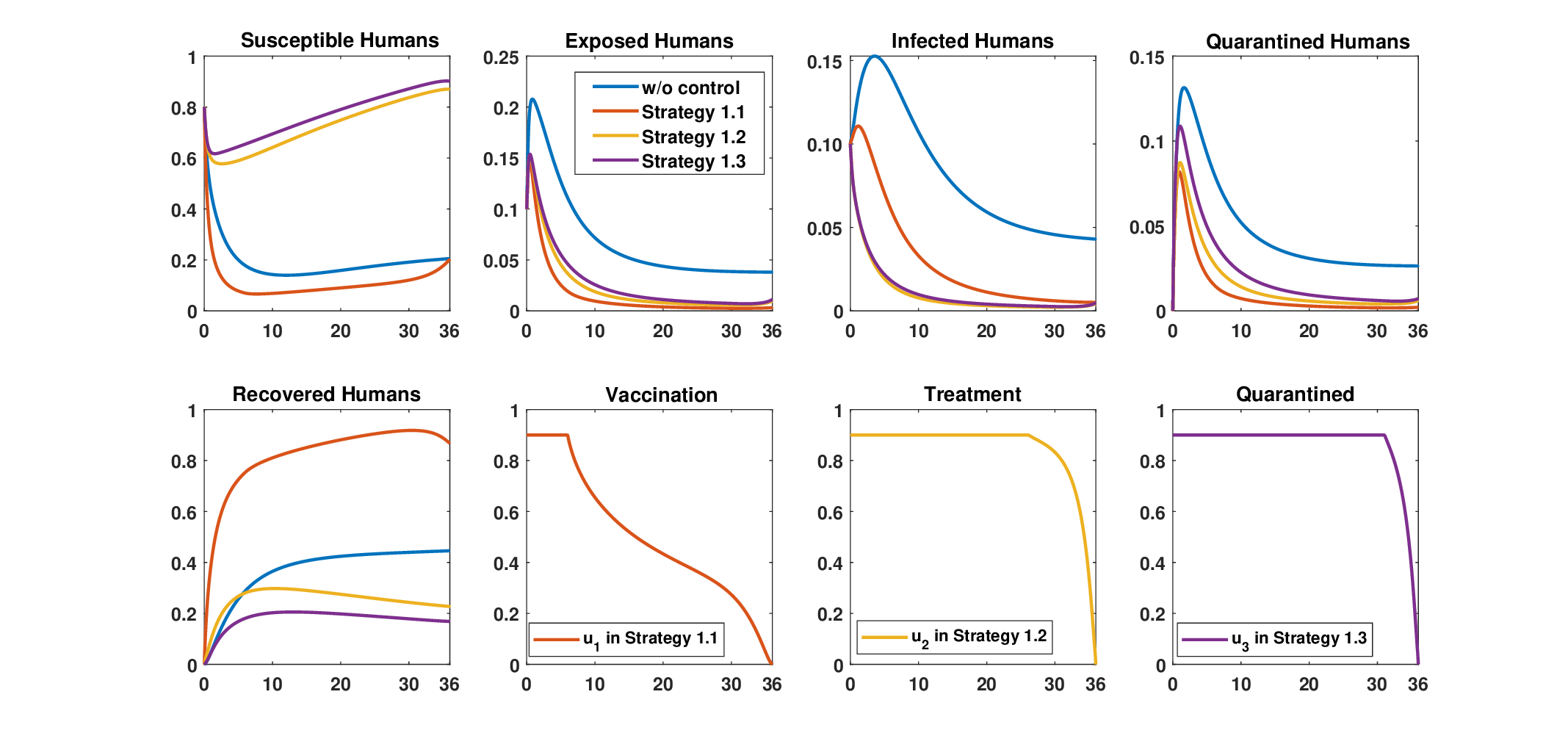} \\[0pt]
\caption{A comparison for Strategy 1: Single control versus uncontrolled
effects for $\alpha=0.90$.}
\label{Fig1}
\end{figure}

\subsection{Strategy~2}

In this strategy, the~results of the double choices among the three controls
to the system and their comparison with the uncontrolled system are
presented in Figure~\ref{Fig2} with the following cases:

\begin{itemize}
\item Strategy 2.1 -- vaccination control $\left( u_{1}\right) $ of susceptible
humans, treatment control $\left( u_{2}\right) $ of infected humans $\left(
u_{1}\neq 0,u_{2}\neq 0,u_{3}=0\right)$;

\item Strategy 2.2 -- vaccination control $\left( u_{1}\right) $ of susceptible
humans, quarantine control $\left( u_{3}\right) $ of infected humans $\left(
u_{1}\neq 0,u_{2}=0,u_{3}\neq 0\right)$;

\item Strategy 2.3 -- treatment control $\left( u_{2}\right) $ of infected
humans, quarantine control $\left( u_{3}\right) $ of infected humans $\left(
u_{1}=0,u_{2}\neq 0,u_{3}\neq 0\right)$.
\end{itemize}

\vspace{-6pt}

\begin{figure}[ht]
\includegraphics[width=.9\linewidth]{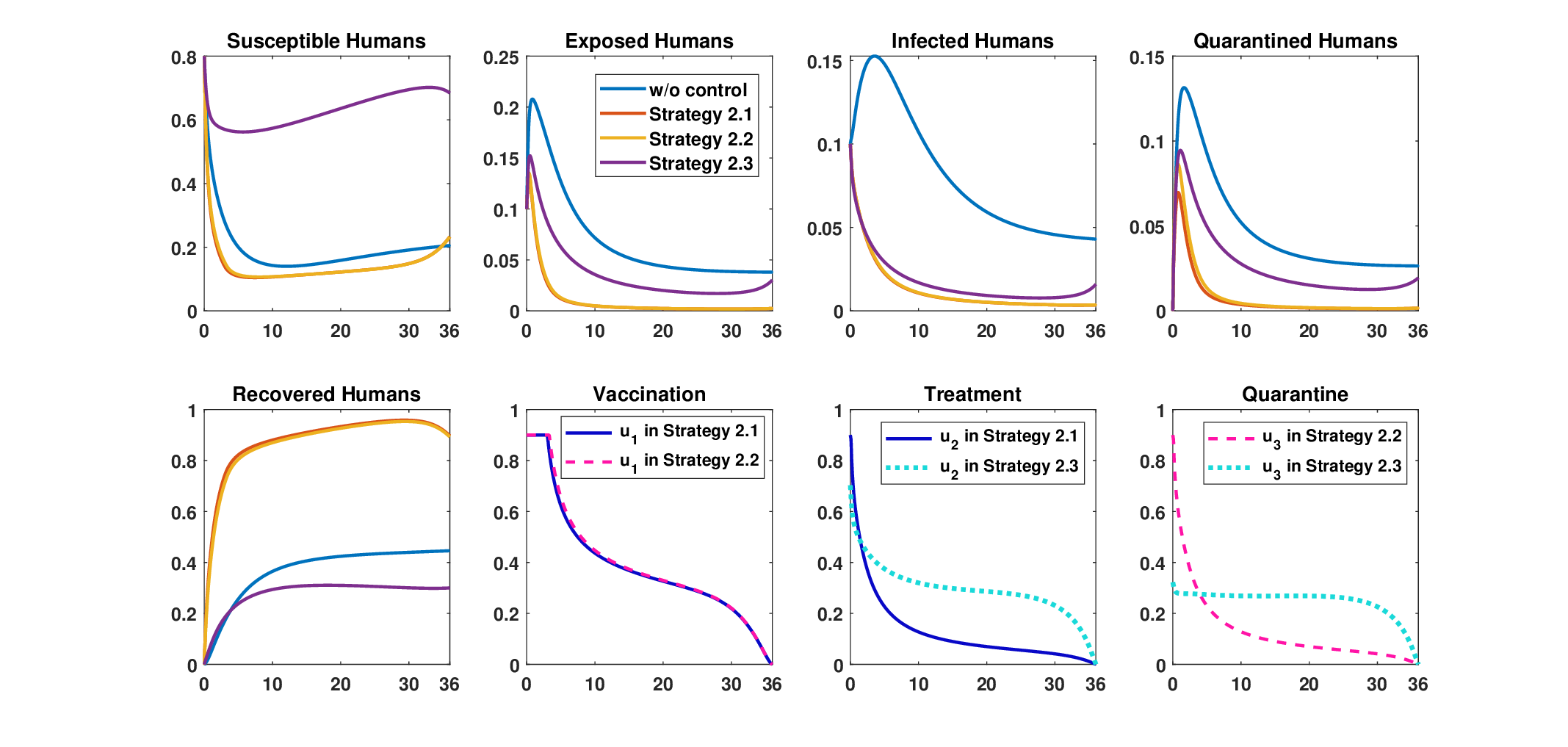} \\[0pt]
\caption{A comparison for Strategy 2: double control versus uncontrolled
effects for $\alpha=0.90$.}
\label{Fig2}
\end{figure}

In Figure~\ref{Fig2}, it is obvious that Strategies 2.1 and 2.2
show similar behavior. Also, these strategies respond better than Strategy
2.3 for each class of population. As~we commented in the previous subsection
Strategy 1, vaccination control greatly improves the results of the strategy
in which it is~incorporated.

\subsection{Strategy~3}

In this strategy, while investigating the effect of triple control on the
system an ideal double control strategy given by Strategy 2 and the
the uncontrolled system is compared. As~seen in Figure~\ref{Fig3},
the simultaneous application of the three controls to the population is
undoubtedly the most effective of all possible situations for infected
humans. However, this result will not be the first policy selection in terms
of the socio-economic structure of the countries. As~the control methods
applied against the disease increase, the~costs will increase, posing
problems for some countries. Therefore, it is desired to find an alternative 
ideal method by comparing the graphs with one of the double controls
obtained in Strategy~2. In~the previous subsection strategy 2, it was
concluded that vaccination-quarantine and vaccination-treatment are
effective strategies for the population. The~quarantine application is a
method that has many harms in terms of both state economy and folk
psychology, as~has been experienced with the COVID-19 epidemic in the recent
past. For~this reason, Strategy 2.1 $\left( u_{1}\neq 0,u_{2}\neq
0,u_{3}=0\right) $, which behaves close to Strategy 2.2, is preferred for~comparison.

\vspace{-3pt}
\begin{figure}[ht]
\includegraphics[width=.9\linewidth]{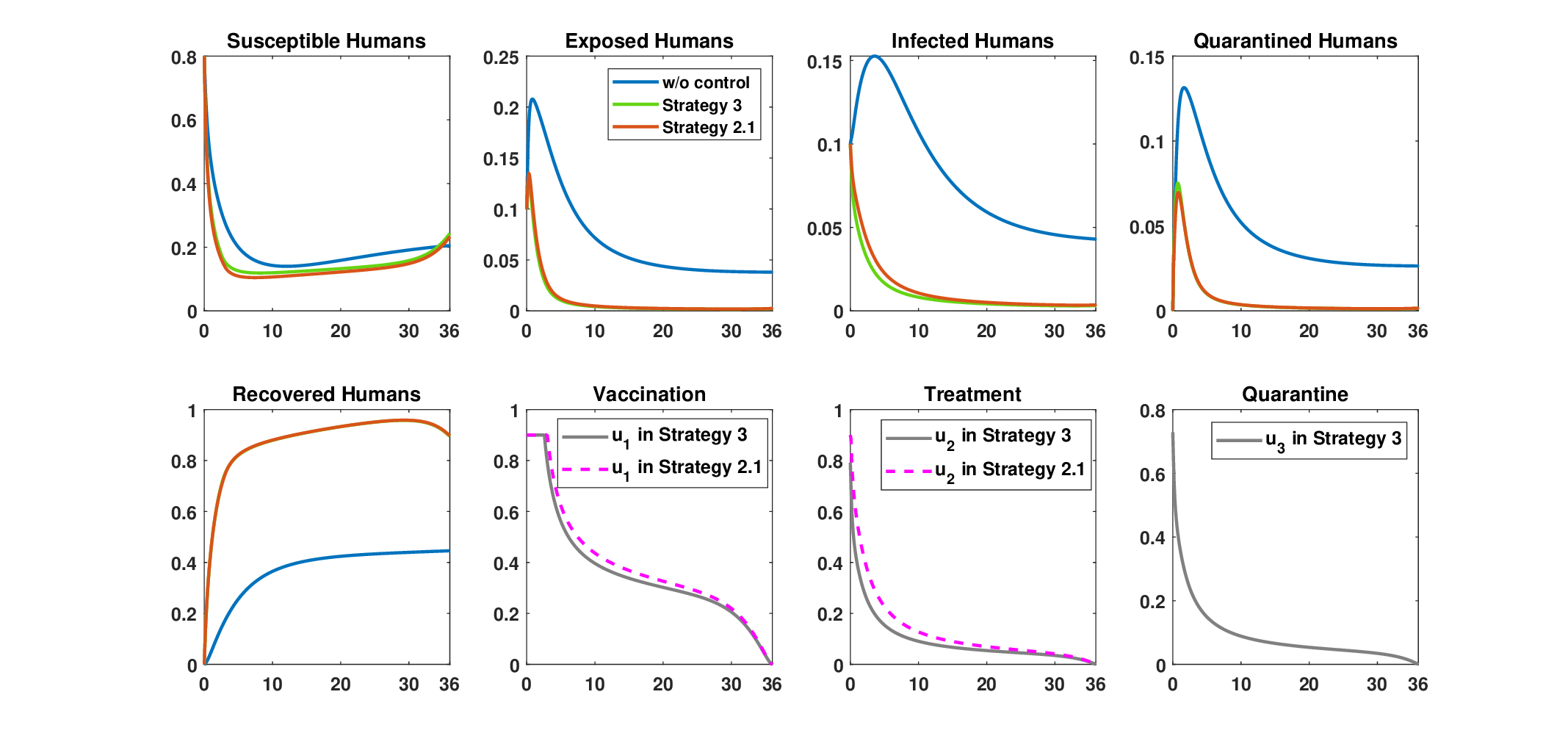} \\[0pt]
\caption{{A} comparison for Strategy 3: Triple control strategy versus the
best double control strategy for $\alpha=0.90$.}
\label{Fig3}
\end{figure}
\unskip


\section{Conclusions}
\label{sec:8}

In recent years, the~rapid spread of the monkeypox virus, especially in
African countries, has increased the efforts to prevent the devastating
effects of the disease. Despite its importance, mathematical modeling and
optimal control studies on the disease are still limited. Motivated by this
need, in~the present study, a~monkeypox model describing the interspecific
spreading of the virus has been discussed in the sense of the
Atangana--Baleanu fractional derivative. In~addition to revealing the
mathematical model of a disease, it is also very important to determine
optimal strategies with parameters that reduce its destructive effect.
Hence, the~main motivation of the present work has been to investigate the
effect of quarantine, treatment, and~vaccination controls on the model. In~
this context, existence-uniqueness results and stability analysis of the
controlled model have been researched. {It has been observed that time-dependent 
quarantine and treatment controls have a reducing effect on the basic 
reproduction number.} The optimality system has been
revealed by Hamiltonian formalism. To~obtain numerical results, the~
Adams-type predictor--corrector method has been implemented. Single, double,
and triple control effects on the model have been illustrated with graphics.
According to the results, it has been seen that dual-control strategies
including vaccination are more effective than dual control composed of
quarantine and treatment. Thus, as~in chickenpox, long-term immunity should
be the primary control strategy in the fight against monkeypox disease. On~
the other hand, as~expected, the~triple control application yields the best
results in declining the number of infected and exposed humans. But,~of
course, at~the onset of the disease, other control strategies will continue
to exert the optimum effect, as~the vaccine is not yet available. As~a
continuation of the study, the~effects of different incidence and treatment
rates on the model are considered for examination. Additionally, the~model can
be developed by considering other interactive diseases with~monkeypox.


\subsection*{Author Contributions}

Conceptualization, E.B. and D.A.; 
methodology, M.Y., D.Y. and D.A.; 
software, M.Y., D.Y. and B.B.\.{I}.E.; 
validation, E.B., D.A., B.B.\.{I}.E. and D.F.M.T.; 
formal analysis, M.Y., D.Y. and D.A.; 
investigation, M.Y., D.Y., E.B., B.B.\.{I}.E., D.A. and D.F.M.T.; 
writing---original draft preparation,  M.Y., D.Y., D.A. and D.F.M.T.; 
writing---review and editing, M.Y., D.Y., D.A. and D.F.M.T.; 
visualization, M.Y. and D.Y.; 
supervision, D.A. and D.F.M.T.;
funding acquisition, D.F.M.T.
All authors have read and agreed to the published version of the manuscript.

\subsection*{Funding}

This research was funded by 
\emph{Funda\c{c}\~{a}o para a Ci\^{e}ncia e a Tecnologia} (FCT), 
grant numbers UID/04106/2025 (\url{https://doi.org/10.54499/UID/04106/2025}) 
and UID/PRR/04106/2025 (\url{https://doi.org/10.54499/UID/PRR/04106/2025}),	
and by FCT project \emph{Mathematical Modelling of Multiscale
Control Systems: Applications to Human Diseases} (CoSysM3), 
reference 2022.03091.PTDC (\url{https://doi.org/10.54499/2022.03091.PTDC}).

\subsection*{Data Availability}

Data are contained within the article.

\subsection*{Conflicts of interest}

The authors declare no conflicts of~interest.



\end{document}